\documentclass[12pt]{amsart}
\usepackage{amsmath,amssymb,amsfonts,amsthm,amsopn}
\usepackage{graphics}
\usepackage{xcolor}
\usepackage{pb-diagram}
\usepackage[title]{appendix}
\usepackage{graphicx}
\usepackage{mathtools}
\mathtoolsset{showonlyrefs}

 \def\Spnr{Sp(d,\R)}
 \def\Gltwonr{GL(2d,\R)}

\newtheorem{theorem}{Theorem}[section]
\newtheorem{lemma}[theorem]{Lemma}
\newtheorem{corollary}[theorem]{Corollary}
\newtheorem{proposition}[theorem]{Proposition}
\newtheorem{definition}[theorem]{Definition}

\newtheorem{remark}[theorem]{Remark}

\newcommand{\beqa}{\begin{eqnarray*}}
\newcommand{\eeqa}{\end{eqnarray*}}

\DeclareMathOperator*{\Sym}{Sym}

\newcommand{\field}[1]{\mathbb{#1}}
\newcommand{\bR}{\field{R}}        
\newcommand{\bN}{\field{N}}        
\newcommand{\bZ}{\field{Z}}        
\def\la{\lambda}

\def\cF{\mathcal{F}}              
\def\cS{\mathcal{S}}
\def\cD{\mathcal{D}}

\def\cA{\mathcal{A}}

\def\cI{\mathcal{I}}
\def\cC{\mathcal{C}}

\def\a{\aleph}
\def\hf{\hat{f}}

\def\rd{\bR^d}

\def\rdd{{\bR^{2d}}}

\def\lrd{L^2(\rd)}
\def\lrdd{L^2(\rdd)}

\def\spdr{{\mathfrak {sp}}(d,\R)}

\def\intrd{\int_{\rd}}
\def\intrdd{\int_{\rdd}}

\def\R{\right)}

\def\<{\left<}
\def\>{\right>}

\def\mv1{M_v^1}

\def\phas{(x,\o )}
\def\mn{(m,n)}
\def\mn'{(m',n')}

\def\Spnr{Sp(d,\R)}

\newcommand{\norm}[1]{\lVert#1\rVert}

\def\o{\eta}
\def\a{\alpha}

\def\N{\mathbb{N}}
\def\R{\mathbb{R}}
\def\Ren{\mathbb{R}^d}

\def\sch{\mathcal{S}}

\def\Sn2{S_{2}(L^{2}(\Ren))}
\def\S1{S_{1}(L^{2}(\Ren))}
\def\sig00{\sigma_{0,0}}

\def\la{\langle}
\def\ra{\rangle}

\begin{document}
	\begin{abstract}
		We present a novel Wigner analysis of Fourier integral operators (FIOs), with a particular focus on Schrödinger propagators derived from quadratic Hamiltonians with bounded perturbations. These perturbations are modeled by a pseudodifferential operator $\sigma(x,D)$, whose symbol belongs to the Hörmander class $S^0_{0,0}(\mathbb{R}^{2d})$. By computing and examining the Wigner kernel of these operators, we uncover important insights into a broader class of FIOs, denoted $FIO(S)$, where $S$ represents the symplectic matrix corresponding to the classical symplectic map.
		
		Our analysis demonstrates both the algebraic structure and Wiener’s property of this class, offering a framework to describe the Wigner kernel of Schrödinger propagators for any time $t \in \mathbb{R}$, including at caustic points. This highlights the robustness and effectiveness of the Wigner approach in solving Schrödinger equations, offering avenues for both theoretical understanding and practical application.
\end{abstract}

\title[Wigner Representation of of Schr\"odinger Propagators]{Wigner Representation of  Schr\"odinger Propagators}

\author{Elena Cordero}
\address{Universit\`a di Torino, Dipartimento di Matematica, via Carlo Alberto 10, 10123 Torino, Italy}
\email{elena.cordero@unito.it}
\author{Gianluca Giacchi}
\address{Universit\'a di Bologna, Dipartimento di Matematica, Piazza di Porta San Donato 5, 40126 Bologna, Italy; University of Lausanne, Switzerland; HES-SO School of Engineering, Rue De L'Industrie 21, Sion, Switzerland; Centre Hospitalier Universitaire Vaudois, Switzerland}
\email{gianluca.giacchi2@unibo.it}
\author{Luigi Rodino}
\address{Universit\`a di Torino, Dipartimento di Matematica, via Carlo Alberto 10, 10123 Torino, Italy}
\email{luigi.rodino@unito.it}

\thanks{}


\maketitle

\section{Introduction}

Building upon Wigner's groundbreaking work \cite{Wigner}, we introduce a novel framework for analyzing Schrödinger propagators through the lens of Wigner representation. Extending this concept, we develop the Wigner localization of Fourier integral operators (FIOs), encompassing Schrödinger operators with phase functions characterized by quadratic forms in phase-space variables. This approach opens up perspectives for addressing  problems in quantum mechanics and partial differential equations.

The standard time-frequency analysis of FIOs, utilizing the Gabor matrix, as detailed in \cite{CGNRJMPA} and subsequent works \cite{wiener8, wiener9, locNC09, CGNRJMP2014}, has intrinsic limitations, primarily the reliance on chosen window functions. Our method circumvents this dependency, though at the cost of relinquishing the smooth decay behavior of the operator kernel. Nevertheless, the Wigner kernel proves advantageous, particularly in overcoming the challenges of caustics, while providing a robust tool for numerical computations.

Despite the extensive literature on classical and time-frequency methods for FIOs, the Wigner representation has remained largely unexplored. Our approach, therefore, presents a new direction that not only addresses unresolved issues but also broadens the scope of Wigner analysis for FIOs. We focus here on Schrödinger equations with bounded perturbations, while also laying the groundwork for more general FIOs with tame phase functions (cf. Definition \ref{def2.1}).

In particular, we concentrate on quadratic Hamiltonians perturbed by a pseudodifferential operator $\sigma(x,D)$, expressed in the Kohn-Nirenberg form:
 \begin{equation}\label{C6kohnNirenberg}
	\sigma(x,D) f(x)=\intrd e^{2\pi i x \xi}\sigma(x,\xi)\hat{f}(\xi)\,d\xi,
\end{equation}
with a symbol $\sigma$ in the  H\"{o}rmander class 
$S^0_{0,0}(\rdd)$, consisting of smooth functions  $\sigma$  on $\rdd$ such that 
\begin{equation}\label{I5}
	|\partial_x^\alpha \partial_\xi^\beta \sigma(x,\xi)|\leq c_{\alpha,\beta},\quad\alpha,\beta\in\N^d,\quad x,\xi\in\rd.
\end{equation}


%
%

We aim at studying  the Wigner
representation  of the
propagator $e^{i t H}$,
\begin{equation}\label{PropH}
	H=a(x,D)+\sigma(x,D),
\end{equation}
providing the solution to the Cauchy problem
\begin{equation}\label{C6intro}
	\begin{cases} i \displaystyle\frac{\partial
			u}{\partial t} +a(x,D)u+\sigma(x,D) u=0,\\
		u(0,x)=u_0(x).
	\end{cases}
\end{equation}
The Hamiltonian  $a(x,D)$ is the quantization of a quadratic form
or more generally a pseudodifferential operator in the Kohn-Nirenberg form
whose symbol $a$ satisfies $\partial^\alpha_{x,\xi}a\in S^0_{0,0}(\rdd)$ for $|\a|\geq2$, whereas
$\sigma(x,D)$ is the Kohn-Nirenberg operator expressing the perturbation mentioned above.

Basic examples are
the cases when $a(x,D)=-\Delta$ and $\sigma(x,D)=0$ which give the free particle or
when the quadratic form $a$ is given by $a(x,\xi)=\xi^2+x^2$ and again  $\sigma=0$  yielding the harmonic oscillator operator
$-\Delta+x^2$.
More generally, Hamiltonians of the type $H=-\Delta+xMx+V(x)$, for $V\in S^0_{0,0}(\rd)$, fall in our study.

Applications to our study include the  special cases of  systems with a magnetic potential and bounded perturbations, studied by Cassano and Fanelli  \cite{CF2015,CF2017}. Here the Hamiltonian is $H=\Delta_A+V(x,t)$, where the (electro)magnetic Laplacian is  $\Delta_A=(\nabla-iA(x))^2$, with magnetic potential given by some coordinate transformation $A:\rd\to\rd$, see Section \ref{7} below for details. 

It is well known (cf., \cite{AS78,bony1,bony2,bony3,CGNRJMPA}) that the  Schr\"odinger propagator $e^{itH}$  expressing the solution of \eqref{C6intro}, for small values of $t$ at least,  is a Fourier integral operator of
type I (see Section $6$ below) represented by
\begin{equation}\label{C61.13}
	(e^{i t H}u_0)(t,x)=\intrd e^{2\pi i \Phi(t,x,\xi)}
	b(t,x,\xi)\widehat{u}_0(\xi)\,d\xi,
\end{equation}
where the phase $\Phi$ is a quadratic form or, more generally, a tame phase. The regularity of the amplitude $b$ is related to the regularity of the perturbation $\sigma$, cf. \cite{helffer84}.
Its structure is linked to the Hamiltonian
field of $a(x,\xi)$.  Namely, consider
\begin{equation}\label{eitA}
	\begin{cases}
		2\pi\dot{x}=-\nabla  _\xi a ( x,\xi) \\
		2\pi \dot{\xi}=\nabla _x a (x,\xi)\\
		x(0)=y,\ \xi(0)=\eta,
	\end{cases}
\end{equation}
(the factor $2\pi$ depends on our normalization of the Fourier transform).
Under our assumptions, the solution
$\chi_t(y,\eta)=(x(t,y,\eta),\xi(t,y,\eta))$ exists for all
$t\in\bR$ and defines a symplectic diffeomorphism
$\chi_t:\,\bR_{y,\eta}^{2d}\to\bR_{x,\xi}^{2d}$ \emph{tame}, for every
fixed $t\in\bR$. When $a$ is quadratic, $\chi_t=S_t$ is a classical linear symplectic map in $Sp(d,\bR)$, the symplectic group. 

Let us introduce the most important character of our study: the Wigner distribution, named after E. Wigner, who introduced it in  1932 \cite{Wigner}. This representation was later  developed by Ville, Cohen and  many other authors, see e.g. \cite{Cohen1,Cohen2,Ville48}.
\begin{definition} Consider $f,g\in\lrd$. 
	The cross-Wigner distribution $W(f,g)$ is
	\begin{equation}\label{CWD}
		W(f,g)(x,\xi)=\intrd f(x+\frac t2)\overline{g(x-\frac t2)}e^{-2\pi i t\xi}\,dt.
	\end{equation} If $f=g$ we write $Wf:=W(f,f)$, the so-called Wigner distribution of $f$.
\end{definition}
Generalizing the Wigner approach in 1932, given a linear operator $T:\,\cS(\rd)\to \cS'(\rd)$, we consider  an operator $K$ on $\cS(\rdd)$ such that 
\begin{equation}\label{I3}
	W(Tf,Tg) = KW(f,g), \qquad f,g\in\cS(\rd)
\end{equation}
and its  kernel $k$ which we call the \emph{Wigner kernel}  of $T$:
\begin{equation}\label{I4}
	W(Tf,Tg)(z) = \intrdd k(z,w) W(f,g)(w)\,dw,\quad z\in\rdd,\quad f,g\in\cS(\rd).
\end{equation}

\begin{figure}
	\centerline{
	\includegraphics[width=\textwidth]{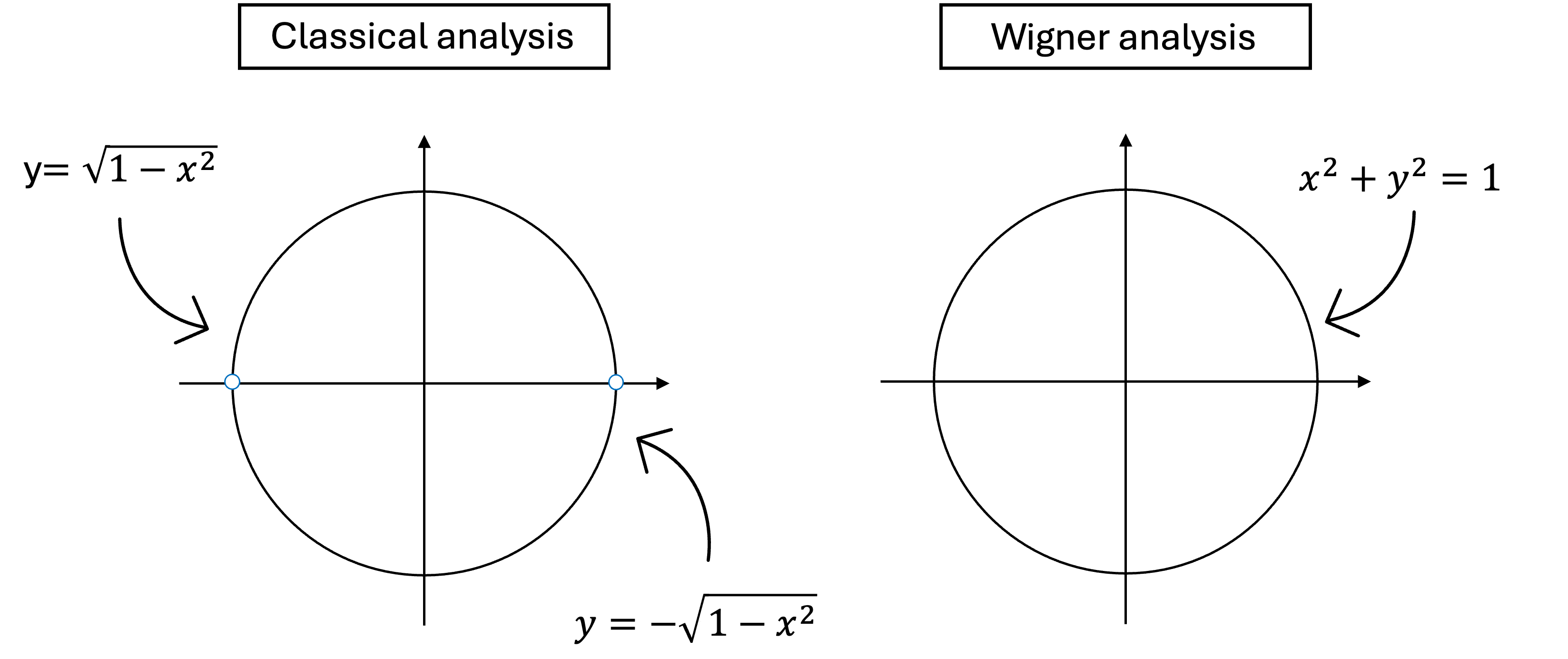}
	}
	\caption{Wigner analysis explained with an example. }
	\label{fig:Wigner}
\end{figure}

Basic example in our study is the {\it Quantum Harmonic Oscillator}, let us give first some intuitive ideas. A useful analogy for understanding Wigner analysis is the representation of the unit circle, as shown in Figure \ref{fig:Wigner}, either as the combination of the smooth one-variable functions $y=\sqrt{1-x^2}$ and $y=-\sqrt{1-x^2}$, for $-1<x<1$, or as the level set $x^2+y^2=1$. The first method, though explicit and involving only one variable, breaks down at the points $(\pm 1, 0)$ and is therefore local. In contrast, the second method, which is implicit and uses two variables, is global. In a similar way, the classical representation of the quantum harmonic oscillator fails at caustic points, making the representation local, as clarified next. Namely, consider the Hamiltonian $a(x,D)=\Delta-x^2$ and $\sigma(x,D)=0$. The solution $u(t,x)$ to \eqref{C6intro} can be expressed as the following FIO of type I \cite{folland89}:
\begin{equation*}
	u(t,x)=(\cos t)^{-d/2}\intrd e^{2\pi i [\frac1{\cos t} x\eta + \frac{\tan t}2 (x^2+\eta^2)]}\hat{u_0}(\eta)\,d\eta,\quad t\not=\frac\pi 2 + k\pi,\,\,k\in\bZ.
\end{equation*}
Here $u(t,x)=T_t$, where $T_t$ is a FIO with symbol $\sigma\equiv 1$ and phase
$$\Phi_t(x,\eta)=\frac1{\cos t} x\eta+
\frac{\tan t}2 (x^2+\eta^2),
$$
for  $t\not=\pi/2+k\pi$, $k\in\mathbb{Z}$. The points  $t=\pi/2+k\pi$ are the \emph{caustic points}. In these points the representation as type I FIO fails.
The canonical transformation $S_t$ is  obtained by solving \eqref{eitA}: 
\begin{equation}\label{C6chiharmonic}
	S_t(y,\eta)=\begin{pmatrix}(\cos
		t)I&(-\sin t)I\\(\sin t)I&(\cos
		t)I\end{pmatrix} \begin{pmatrix}y\\\eta\end{pmatrix},\quad\forall t\in\bR.
\end{equation}
The Wigner kernel of the FIO $T_t$ is 
\begin{equation*}
	k(t,z,w)=\delta_{z-S_t w},\quad w,z\in\rdd, 
\end{equation*}
for every $t\in\bR$, also meaningful in the caustic points.

Other examples are provided by  Schr\"{o}dinger equations with uniform magnetic potential, cf.  Section \ref{7} below.

Through the phase-space perspective of Wigner analysis, a global representation of FIOs with quadratic phase can be achieved, which remains valid even at caustic points. This global approach consists of replacing the conventional operator kernels, given by the Schwartz kernel theorem, with the  Wigner kernels, whose general properties will be discussed in Section \ref{Section3}.\\


As aforementioned, we only consider the case of $a(x,D)$ in \eqref{C6intro} to be the quantization of a quadratic form, so that the phase $\Phi$ is quadratic and  $\chi=S$ is  a linear symplectic map in $Sp(d,\bR)$. The perturbation 
$\sigma$ is chosen in the H\"{o}rmander class $S^0_{0,0}(\rdd)$. This leads to the following general class of FIOs.

\begin{definition}\label{C6T1.1old}
	Consider  $S \in Sp(d,\bR)$.  We say that the operator $T$ in \eqref{I3} is in the class FIO($S$) if its Wigner kernel $k$ in \eqref{I4} satisfies 
	\begin{equation}\label{nucleoFIO}
		k(z,w) =h(z,Sw),\quad z,w\in \rdd,
	\end{equation}
	where $h$ is the kernel of a pseudodifferential operator with symbol in the H\"ormander class
	$ S^0_{0,0}(\bR^{4d})$.
\end{definition}
The quantization does not affect \eqref{nucleoFIO}, since the H\"ormander class property of the symbol does not depend on it \cite[Proposition 1.2.5 and Remark 1.2.6]{NR}.
We refer to \cite{CGRPartII2022} for an alternative approach in terms of metaplectic operators. \par

In the last section, we shall prove that $e^{it H}\in FIO(S_t)$. 
In fact, for $t$ sufficiently small
our assumptions yield $\left|\det \frac{\partial x}{\partial
	y}(t,y,\eta)\right|>0$ in the expression of $S_t$ (see condition \eqref{detcond2} below), and
we shall show that the solution in \eqref{C61.13} satisfies the expression \eqref{nucleoFIO}
with the phase $\Phi(t,x,\xi)$ quadratic for $t$ in a small time interval $(-T,T)$.
In the
classical approach, cf.\ \cite{AS78}, the occurrence of caustics
makes the validity of \eqref{C61.13} local in time. So for $t\in
\bR$ one is led to multiple compositions of local representations,
with unbounded number of variables possibly appearing in the
expression. Whereas $k(t,z,w)$ obviously exists for every
$t\in\bR$, and the equality in \eqref{nucleoFIO} reads as
\begin{equation}\label{nucleoFIOt}
	k(t,z,w) =h_t(z, S_t(w)),\quad z,w\in \rdd,
\end{equation}
holding for every $t\in\bR$.
Fundamental step of the above claim is the \emph{algebra property} of the class FIO($S$). 
More generally, we will prove the following issues: 
\begin{theorem}[Properties of the class FIO($S$)] \label{tc1old}\par
	
	(i) An operator $T\in  FIO(S)$ is bounded on $\lrd$.
	
	(ii)  If $T_i\in FIO(S_i)$, $i=1,2$, then $T_1 T_2\in
	FIO(S_1 S_2)$.
	
	(iii)  If $T\in FIO(S)$ is invertible on $L^2(\rd)$,
	then $T^{-1}\in FIO(S^{-1})$.
\end{theorem}

Also, a general theory of  Wigner kernels  is missing in literature, and we need a preliminary study on that. 
\begin{theorem}[Properties of the Wigner kernel] \label{Wkernel} 
Let $T:\cS(\rd)\to\cS'(\rd)$ have kernel $k_T\in\cS'(\rdd)$ and Wigner kernel $k\in\cS'(\bR^{4d})$ defined in \eqref{I4}. \\
(i) The Wigner kernel $k$ satisfies
$$\la k,Wf\otimes Wg\ra_{\la\cS',\cS\ra}\geq0\quad\mbox{for\,\,all}\,\,f,g\in\cS(\rd).$$

(ii) Set $\mathfrak{T}_pF(x,\xi,y,\eta)=F(x,y,\xi,-\eta)$. Then
\begin{equation*}
	k=\mathfrak{T}_pWk_T.
\end{equation*}
 As a result, \\
(ii.1) $k\in L^2(\bR^{4d})$ if and only if $k_T\in L^2(\rdd)$.\\
(ii.2) $k\in\cS(\bR^{4d})$ if and only if $k_T\in\cS(\rdd)$.
\end{theorem}

Observe that this study could find further applications in Hardy's uncertainty principles, where the Wigner distribution has become the crucial tool, cf. the recent contributions \cite{FM2021,Helge} and references therein.
\vspace{0.3truecm}

The paper is organized as follows. Section \ref{sec:preliminaries} contains the preliminaries on time-frequency analysis and metaplectic operators which will be used in our study. Section \ref{Section3} is devoted to the properties of the Wigner kernel. In particular, Theorem \ref{Wkernel}  will be proved. Section \ref{sec:Section4} exhibits the properties of the class $FIO(S)$ and proves Theorem \ref{tc1old}. Section \ref{sec:Section5} treats in detail the Wigner kernel of FIOs of type I, which is the FIO of the Schr\"{o}dinger propagator \eqref{C61.13}. For further studies we treat the general case of FIOs with \emph{tame} phases, cf. Definition \ref{def2.1} below. 
Section \ref{sec:Section7} contains the application to  Schr\"{o}dinger propagators announced above. The  examples of the free particle and the uniform magnetic potential are detailed in Section \ref{7}. In Section \ref{sec:Section8} we discuss the comparison of our method within the existing literature. Finally, we devote the Appendix to a commutation relation that involves the partial Fourier transform with respect to the second variable and linear changes of variables.
\par

We conjecture our approach can treat more general cases, in particular Hamiltonians $a(x,\xi)$ homogeneous of order $2$ with respect to $w=(x,\xi)$. Definition \ref{C6T1.1old} should work, but in 
\eqref{nucleoFIO}  the kernel $h$ should refer to a somewhat exotic class of pseudodifferential operators, as suggested by Theorem \ref{T4.1} below.
\section{Preliminaries}\label{sec:preliminaries}
\textbf{Notation.} We call $t^2=t\cdot t$,  $t\in\rd$, and
$xy=x\cdot y$.  The space   $\sch(\Ren)$ is the Schwartz class and $\sch'(\Ren)$  the space of temperate distributions.   The brackets  $\la f,g\ra_{\la\cS',\cS\ra}$ denote the extension to $\sch' (\Ren)\times\sch (\Ren)$ of the inner product $\la f,g\ra_{L^2(\Ren)}=\int f(t){\overline {g(t)}}dt$ on $L^2(\Ren)$ (conjugate-linear in the second component). 
We denote $GL(d,\R)$ the group of real invertible $d\times d$ matrices and with $\Sym(d,\bR)$ the group of $d\times d$ real symmetric matrices. 
The Fourier transform is defined by $$\cF f(\xi)=\hf(\xi)=\intrd f(x)e^{-2\pi i x\xi}\,dx.$$
\subsection{The symplectic group $Sp(d,\mathbb{R})$ and the metaplectic operators \cite{Gos11}}
The standard symplectic matrix is denoted by
\begin{equation}\label{J}
	J=\begin{pmatrix} 0&I\\-I&0\end{pmatrix}.
\end{equation}
The symplectic group is given by
\begin{equation}\label{defsymplectic}
	\Spnr=\left\{S\in\Gltwonr:\;S^T JS=J\right\},
\end{equation}
where  $S^T$ is the transpose of $S$. We have $\det(S)=1$. If $S\in Sp(d,\bR)$, we write its $d\times d$ block decomposition: 
\begin{equation}\label{S}
	S=\begin{pmatrix}
	A & B\\
	C & D
	\end{pmatrix},
\end{equation}
where $A,B,C,D\in\bR^{d\times d}$. We say that $S$ is \emph{upper triangular} if $C=0$.

For $E\in GL(d,\bR)$ and $C\in \Sym(d,\bR)$, define:
\begin{equation}\label{defDLVC}
	\cD_E:=\begin{pmatrix}
		E^{-1} & 0\\
		0 & E^T
	\end{pmatrix} \qquad \text{and} \qquad V_C:=\begin{pmatrix}
		I & 0\\ C & I
	\end{pmatrix}.
\end{equation}
The matrices $J$, $V_C$, and $\cD_E$  generate the group $Sp(d,\bR)$.\\
The symplectic  algebra 
$\spdr$ is the algebra of
$2d\times 2d$ real matrices
$A$   such that $e^{t A}
\in \Spnr$, for every $t\in\R$.\par

The metaplectic group, named $Mp(d,\bR)$,
is a realization of the two-fold covering of $Sp(d,\bR)$. The projection
\begin{equation}\label{piMp}
	\pi^{Mp}:Mp(d,\bR)\to Sp(d,\bR)
\end{equation} is a group homomorphism with kernel $\ker(\pi^{Mp})=\{-id_{{L^2}},id_{{L^2}}\}$.

Here, if $\hat S\in Mp(d,\bR)$, the matrix $S$ will always be the unique symplectic matrix satisfying $\pi^{Mp}(\hat S)=S$.
A metaplectic operator we are going to use is the dilation operator. Namely, if \begin{equation}\label{taue}
	\mathfrak{T}_E:=|\det(E)|^{1/2}\,f(E\cdot),
\end{equation} then $\pi^{Mp}(\mathfrak{T}_E)=\cD_E$. Let us recall also that $\pi^{Mp}(\cF)=J$ and $\cA_{FT2}=\pi^{Mp}(\cF_2)\in Sp(2d,\bR)$, where
\begin{equation}\label{AFT2}
	\cA_{FT2}=\begin{pmatrix}
		I & 0 & 0& 0 \\
		0 & 0 & 0& I\\
		0 & 0 & I & 0 \\
		0 & -I & 0 & 0
	\end{pmatrix}
\end{equation}
and $\cF_2$ denotes the partial Fourier transform with respect to the second variable:
\begin{equation}\label{cF2}
	\cF_2 F(x,\xi)=\intrd e^{-2\pi i y\xi}F(x,y)\,dy,\quad F\in\cS(\rdd).
\end{equation}
	
%

Any metaplectic operator $\hat S\in Mp(d,\bR)$ is bounded on $\cS(\rd)$, and it can be extended to a bounded operator on $\cS'(\rd)$ by: 
\begin{equation}\label{MoyalDistr}
	\la\hat S f,\hat S g\ra_{\la\cS',\cS\ra}=\la f,g\ra_{\la\cS',\cS\ra}, \qquad f\in\cS'(\rd), g\in\cS(\rd).
\end{equation}
\subsection{Metaplectic Wigner distributions}
In what follows we will make use of few techniques from the theory of metaplectic Wigner distributions. We collect here the basic elements needed in this study.
For $\hat\cA\in Mp(2d,\bR)$, the \emph{metaplectic Wigner distribution} associated to $\hat\cA$ is defined  as
\begin{equation}\label{WA}
W_\cA(f,g)=\hat\cA(f\otimes\bar g),\quad f,g\in L^2(\rd).
\end{equation}
Many popular time-frequency representations are metaplectic Wigner distributions.

Recall the  $\tau$-Wigner distributions, $\tau\in\bR$, defined by
\begin{equation}\label{tauWigner}
	W_\tau(f,g)(x,\xi)=\int_{\rd} f(x+\tau t)\overline{g(x-(1-\tau)t)}e^{-2\pi i\xi  t}dt, \qquad  (x,\xi)\in\rdd,
\end{equation}
for $f,g\in L^2(\rd)$.
The cases $\tau=0$ and $\tau=1$ are the so-called (cross-)\textit{Rihacek distribution}
\begin{equation}\label{RD}
	W_0(f,g)(x,\xi)=f(x)\overline{\hat g(\xi)}e^{-2\pi i\xi x}, \quad (x,\xi)\in\rdd,
\end{equation}
and (cross-)\textit{conjugate Rihacek distribution}
\begin{equation}\label{CRD}
	W_1(f,g)(x,\xi)=\hat f(\xi)\overline{g(x)}e^{2\pi i\xi x}, \quad (x,\xi)\in\rdd.
\end{equation}
The case  $\tau=1/2$ is the cross-Wigner distribution, defined in \eqref{CWD}.
$\tau$-Wigner distributions are metaplectic Wigner distributions:
 $$W_\tau(f,g)=\hat A_\tau(f\otimes\bar g),$$ where
\begin{equation}\label{Atau}
	A_\tau=\begin{pmatrix}
		(1-\tau)I & \tau I & 0 & 0\\
		0 & 0 & \tau I & -(1-\tau)I\\
		0 & 0 & I & I\\
		-I & I & 0 & 0
	\end{pmatrix}.
\end{equation}
In particular, for the Wigner case $\tau=1/2$,
\begin{equation}\label{A12}
Wf=W_{1/2}(f,f)=\hat A_{1/2}(f\otimes\bar f), \quad f\in L^2(\rd).
\end{equation}
Observe that $\hat A_{1/2}$ can be split into the product
\begin{equation}\label{A12prodotto}
	\hat A_{1/2}=\cF_2\mathfrak{T}_L,
\end{equation}
where
	\[
L=\begin{pmatrix}I & \frac{1}{2}I\\ I & -\frac{1}{2}I\end{pmatrix}
\]
and $\mathfrak{T}_L$ is the dilation operator in \eqref{taue} with $E=L$.
	From \eqref{A12prodotto} we obtain
	\[
	\hat A_{1/2}F(x,\xi)=\int_{\rd}F(x+t/2,x-t/2)e^{-2\pi i\xi t}dt, \qquad F\in \cS(\rdd).
	\]
	Moreover, $\hat A_{1/2}^{-1}=\mathfrak{T}_{L^{-1}}\cF_2^{-1}$, where
	\[
	L^{-1}=\begin{pmatrix} \frac{1}{2}I& \frac{1}{2}I\\
		I & -I
	\end{pmatrix},
	\]
	so that
\begin{equation}\label{rem1}
	\hat A_{1/2}^{-1}F(x,\xi)=\int_{\rd}F(x/2+\xi/2,y)e^{2\pi i(x-\xi) y}dy, \qquad F\in\cS(\rdd).
\end{equation}

Similarly to the Wigner case, $W_\cA$ enjoys  the following properties:
\begin{proposition}\label{prop25}
	Let $W_\cA$ be a metaplectic Wigner distribution. Then, the following continuity properties hold:\\
	(i) $W_\cA:L^2(\rd)\times L^2(\rd)\to L^2(\rdd)$ is bounded.\\
	(ii) $W_\cA:\cS(\rd)\times\cS(\rd)\to \cS(\rdd)$ is continuous.\\
	(iii) $W_\cA:\cS'(\rd)\times\cS'(\rd)\to\cS'(\rdd)$ is continuous.\\
	(iv) Moyal's identity holds: for every $f,g,\varphi,\gamma\in L^2(\rd)$,
	\[
		\la W_\cA(f,g),W_\cA(\varphi,\gamma)\ra_{L^2(\rdd)}=\la f,\varphi\ra_{L^2(\rd)}\overline{\la g,\gamma\ra}_{L^2(\rd)}.
	\]
	(v) For $f\in\cS'(\rd)$ and $g\in\cS(\rd)$, setting $W_\cA f=W_\cA(f,f)$, and similarly for $W_\cA g$, we have:
	\begin{equation}\label{MoyalNext}
		\la W_\cA f,W_\cA g\ra_{\la\cS',\cS\ra}=|\la f,g\ra_{\la\cS',\cS\ra}|^2.
	\end{equation}
\end{proposition}

%

\subsection{Pseudodifferential operators}
For a given time-frequency representation $Q$ and distribution $\sigma\in\cS'(\rdd)$, we consider the association:
\[
	L_Q:\sigma\in\cS'(\rdd)\mapsto Op_Q(\sigma),
\]
where $Op_Q(\sigma):\cS(\rd)\to\cS'(\rd)$ is defined by:
\[
	\la Op_Q(\sigma)f,g\ra_{\la\cS',\cS\ra}:=\la \sigma,Q(g,f)\ra_{\la\cS',\cS\ra}, \qquad f,g\in\cS(\rd).
\]
The operator $Op_Q(\sigma)$ is called \emph{pseudodifferential operator} with \emph{symbol} $\sigma$ and  \emph{quantization} $Q$.

The most popular correspondences are the \emph{Weyl correspondence} (when $Q$ is the cross-Wigner distribution $W$) and the \emph{Kohn-Nirenberg correspondence} (when $Q$ is the cross-Rihacek distribution $W_0$) for which we used the notation $\sigma(x,D)=Op_{W_0}(\sigma)$. 
The choice $Q=W_1$ gives the so-called adjoint correspondence. If $\sigma\in\cS'(\rdd)$ is the symbol of the Kohn-Nirenberg quantization of a pseudodifferential operator $\sigma(x,D)$, the kernel $h$ of $\sigma(x,D)$ can be written in terms of $\sigma$ as:
\[
	h(x,\xi)=\cF_2\sigma(x,\xi-x), \qquad x,\xi\in\rd.
\]

\subsection{Tame phase functions and related canonical transformations}\label{sec24}
In what follows we recall the phase functions and related canonical transformations that intervene in this study. We  follow the notation of \cite{CGNRJMPA}.
\begin{definition}\label{def2.1}  A real  phase function  $\Phi$  on $\rdd$ is called \emph{tame} if it satisfies the following
properties:\\
{\it  A1.} $\Phi\in \cC^{\infty}(\rdd)$;\\
{\it  A2.} for $z=\phas\in\rdd$,
\begin{equation}\label{phasedecay}
	|\partial_z^\a \Phi(z)|\leq
	C_\a,\quad \forall \a\in \bN^{2d},\,\,|\a|\geq
	2;\end{equation}
{\it  A3.} there exists $\delta>0$ such
that
\begin{equation}\label{detcond}
	|\det\,\partial^2_{x,\eta} \Phi(x,\o)|\geq \delta,\quad\forall x,\eta\in\rd.
\end{equation}
\par
Solving the system
\begin{equation}\label{cantra} \left\{
	\begin{array}{l}
		y=\Phi_\eta(x,\eta)
		\\
		\xi=\Phi_x(x,\eta), \rule{0mm}{0.55cm}
	\end{array}
	\right.
\end{equation}
 with respect to $(x,\xi)$, we obtain a
map $\chi:\rdd\to\rdd$
defined by \begin{equation}\label{chi}
	(x,\xi)=\chi(y,\o),
\end{equation}
 with the following properties:\\
 
 \noindent {\it  A4.} $\chi$ is smooth, invertible,  and
 preserves the symplectic form in $\rdd$, i.e., $$dx\wedge d\xi= d
 y\wedge d\eta$$
  ($\chi $ is a \emph{symplectomorphism}).
 \\
 {\it  A5.} For $z=(y,\eta)$,
 \begin{equation}\label{chistima}
 	|\partial_z^\a \chi(z)|\leq C_\a,\quad\forall \a\in \bN^{2d},\,\, |\a|\geq 1;\end{equation}
 {\it A6}. There exists $\delta>0$ such that, for
 $(x,\xi)=\chi(y,\eta)$,
 \begin{equation}\label{detcond2}
 	\left|\det\,\frac{\partial x}{\partial y}(y,\eta)\right|\geq \delta.
 \end{equation}
\end{definition}
 Conversely,  as it was observed in \cite{CGNRJMPA}, to every transformation $\chi$ satisfying the three hypothesis above corresponds a tame phase $\Phi$, uniquely
 determined up to a constant. \par

 In our study we will consider linear symplectic transformations $S\in Sp(d,\R)$ with the $d\times d$ block decomposition \eqref{S}.
%
 Then, condition {\it A6}  of Definition \ref{def2.1} is
 equivalent to $\det A\not=0$. In this case, the related metaplectic operator $\hat{S}$ is explicitly
 given by the FIO of type I
 \begin{equation}\label{f4}
 \hat{S}f(x)=(\det A)^{-1/2}\int _{\rd }  e^{2\pi i \Phi(x,\eta)}
 	\hat{f}(\eta)\,d\eta \,
 \end{equation}
 with  the   phase $\Phi$  given by
 \begin{equation}\label{fase}\Phi(x,\eta)=\frac12  x CA^{-1}x+
 	\eta  A^{-1} x-\frac12\eta  A^{-1}B\eta \, .
 \end{equation}
 (see  \cite[Proposition 6.1.18]{Elena-book} or \cite[Theorem 4.51]{folland89}).

\section{Properties of the Wigner Kernel}\label{Section3}

This section will be mainly focused on the general properties of Wigner kernels. First, we give the precise definition of the Wigner kernel of a linear continuous operator $T:\cS(\rd)\to\cS'(\rd)$, then we shall prove its existence and uniqueness.

\begin{definition}\label{defWKernel}
	Let $T:\cS(\rd)\to\cS'(\rd)$ be a continuous linear operator. The \emph{Wigner kernel} of $T$ is the distribution $k\in\cS'(\bR^{4d})$ such that:
	\begin{equation}\label{kerFormula}
		\la W(Tf,Tg),W(u,v)\ra_{\la\cS',\cS\ra}=\la k,W(u,v)\otimes\overline{W(f,g)}\ra_{\la\cS',\cS\ra}, \qquad f,g,u,v\in\cS(\rd).
	\end{equation}
\end{definition}

	We observe that if $k\in\cS(\bR^{4d})$, we retrieve the integral formula \eqref{I4}. 

%
%
Moreover, for every $f,g,u,v\in\cS(\rd)$,
\begin{align*}
	 \la W(Tf,Tg), W(u,v)\ra_{\la\cS',\cS\ra}&=\la \hat{A}_{1/2}(Tf\otimes\overline{Tg}), \hat{A}_{1/2}(u\otimes\overline{v})\ra_{\la\cS',\cS\ra}\\
	&=\la Tf,u\ra_{\la\cS',\cS\ra}\overline{\la Tg,v\ra}_{\la\cS',\cS\ra}.
\end{align*}
\begin{remark}
	In particular, from the above identities, we retrieve Theorem \ref{Wkernel} $(i)$. In fact, for every $f,g\in \cS(\rd)$,
\begin{align*}
	\langle k, Wf\otimes Wg\rangle_{\la\cS',\cS\ra} &=\langle Tg,f\rangle_{\la\cS',\cS\ra} \overline{\la Tg,f\ra}_{\la\cS',\cS\ra}=|\la Tg,f\ra_{\la\cS',\cS\ra}|^2\geq0.
\end{align*}
\end{remark}



We need the following lemma.
\begin{lemma}\label{lemmaDensity}
	The set $$ \mbox{span}\{W(f_1,g_1)\otimes \overline{W(f_2,g_2)}, f_1,f_2,g_1,g_2\in\cS(\rd)\} $$ is dense in $\cS(\bR^{4d})$. 
\end{lemma}
\begin{proof}
	By \cite[Lemma 1.1]{AEFN06}, $\mbox{span}\{W(f,g):f,g\in\cS(\rd)\}$ is dense in $\cS(\rdd)$. The assertion follows by the well-known density of $\mbox{span}\{F\otimes G:F,G\in\cS(\rdd)\}$ into $\cS(\bR^{4d})$.
\end{proof}

We left open the question of whether the Wigner kernel exists for every continuous linear operator $T$. In the following result, we prove that the Wigner kernel $k$ of such an operator $T$ exists and it is unique, and we relate it to the kernel of $T$ given by Schwartz' kernel theorem, i.e. the unique tempered distribution $k_T$ characterized by:
\[
\la Tf,g\ra_{\la\cS',\cS\ra}=\la k_T,g\otimes\bar f\ra_{\la\cS',\cS\ra}, \qquad f,g\in\cS(\rd).
\]

\begin{theorem}\label{3.3}
	Let $T:\cS(\rd)\to\cS'(\rd)$ have kernel $k_T\in\cS'(\rdd)$. There exists a unique distribution $k\in\cS'(\bR^{4d})$ such that \eqref{kerFormula} holds. Therefore, every continuous linear operator $T:\cS(\rd)\to\cS'(\rd)$ has a unique Wigner kernel. Moreover,
	\begin{equation}\label{nuclei}
		k=\mathfrak{T}_pWk_T,
	\end{equation}
	where $\mathfrak{T}_pF(x,\xi,y,\eta)=F(x,y,\xi,-\eta)$.
\end{theorem}
\begin{proof}
Let us compute the unitary operator $\hat\cA_W$ such that $$\hat\cA_W((f_1\otimes \bar g_1)\otimes(f_2\otimes\bar g_2))=W(f_1,g_1)\otimes W(f_2,g_2),$$ for every $f_1,f_2,g_1,g_2\in L^2(\rd)$. For, let $f_1,f_2,g_1,g_2\in L^2(\rd)$, then:
	\begin{align*}
		W(f_1,g_1)\otimes W(f_2,g_2)&=\cF_2\mathfrak{T}_L(f_1\otimes\bar g_1)\otimes\cF_2\mathfrak{T}_L(f_2\otimes\bar g_2)\\
		&=\cF_{2,4}\mathfrak{T}_L'(f_1\otimes\bar g_1\otimes f_2\otimes\bar g_2),
	\end{align*}
	where $\cF_{2,4}F(x,\xi,y,\eta)=\int_{\rdd}F(x,t,y,r)e^{-2\pi i(t\xi+r\eta)}dtdr$ is the partial Fourier transform with respect to the second and fourth variables, and 
	$\mathfrak{T}_L'F(x,\xi,y,\eta)=F(x+\xi/2,x-\xi/2,y+\eta/2,y-\eta/2)$. 
	Therefore,
	\begin{equation}\label{DefAW}
	\hat\cA_W=\cF_{2,4}\mathfrak{T}_L'.
	\end{equation}
	
	For $f_1,f_2,g_1,g_2\in\cS(\rd)$, we have:
	\begin{align*}
		\la Tf_1,g_1 \ra_{\la\cS',\cS\ra}\overline{\la Tf_2,g_2\ra}_{\la\cS',\cS\ra}&=\la k_T,g_1\otimes\bar f_1\ra_{\la\cS',\cS\ra}\overline{\la k_T,g_2\otimes\bar f_2 \ra}_{\la\cS',\cS\ra}\\
		&=\la k_T\otimes\overline{{k_T}},g_1\otimes\bar f_1\otimes \bar g_2\otimes f_2\ra_{\la\cS',\cS\ra}\\
		&=\la k_T\otimes\overline{k_T},\mathfrak{T}_{2,3}(g_1\otimes\bar g_2\otimes \bar f_1\otimes  f_2)\ra_{\la\cS',\cS\ra}\\
		&=\la \mathfrak{T}_{2,3}(k_T\otimes\overline{k_T}),g_1\otimes\bar g_2\otimes \bar f_1\otimes f_2\ra_{\la\cS',\cS\ra}\\
		&=\la \widehat{\cA}_W\mathfrak{T}_{2,3}(k_T\otimes\overline{k_T}),W(g_1,g_2)\otimes W(\bar f_1,\bar f_2)\ra_{\la\cS',\cS\ra}\\
		&=\la \widehat{\cA}_W\mathfrak{T}_{2,3}(k_T\otimes\overline{k_T}),W(g_1,g_2)\otimes \mathcal{I}_2\overline{W( f_1, f_2)}\ra_{\la\cS',\cS\ra}\\
		&=\la \mathfrak{T}\widehat{\cA}_W\mathfrak{T}_{2,3}(k_T\otimes\overline{k_T}),W(g_1,g_2)\otimes \overline{W( f_1, f_2)}\ra_{\la\cS',\cS\ra},
	\end{align*}
	where $\mathfrak{T}_{2,3}F(x,\xi,y,\eta)=F(x,y,\xi,\eta)$, $\cI_2F(x,\xi)=F(x,-\xi)$ and $\mathfrak{T}F(x,\xi,y,\eta)=F(x,\xi,y,-\eta)$. Therefore, $k:=\mathfrak{T}\widehat{\cA}_W\mathfrak{T}_{2,3}(k_T\otimes\overline{k_T})$ is a Wigner kernel for $T$. We prove that $k$ satisfies \eqref{nuclei}: 
	\begin{align*}
		k(x,\xi,y,\eta)&=\mathfrak{T}\cF_{2,4}\mathfrak{T}_L'\mathfrak{T}_{2,3}(k_T\otimes \bar k_T)(x,\xi,y,\eta)\\
		&= \cF_{2,4}\mathfrak{T}_L'\mathfrak{T}_{2,3}(k_T\otimes \bar k_T)(x,\xi,y,-\eta) \\
		&=\int_{\rdd}\mathfrak{T}_L'\mathfrak{T}_{2,3}(k_T\otimes\overline{k_T})(x,t,y,r)e^{-2\pi i(t \xi-r  \eta)}dtdr\\
		&=\int_{\rdd}\mathfrak{T}_{2,3}(k_T\otimes\overline{k_T})(x+t/2,x-t/2,y+r/2,y-r/2)e^{-2\pi i(t \xi-r  \eta)}dtdr\\
		&=\int_{\rdd}(k_T\otimes\overline{k_T})(x+t/2,y+r/2,x-t/2,y-r/2)e^{-2\pi i(t \xi-r  \eta)}dtdr\\
		&=\int_{\rdd}k_T(x+t/2,y+r/2)\overline{k_T(x-t/2,y-r/2)}e^{-2\pi i(t \xi-r  \eta)}dtdr\\
		&=Wk_T(x,y,\xi,-\eta).
	\end{align*}
	This gives \eqref{nuclei}. It remains to prove the uniqueness of the Wigner kernel. For, let $k_1,k_2\in\cS'(\bR^{4d})$ be two Wigner kernels of $T$. Then, for every $f,g,u,v\in\cS(\rd)$,
	\begin{align*}
		\la k_1-k_2,W(u,v)\otimes\overline{W(f,g)}\ra_{\la\cS',\cS\ra}&=\la W(Tf,Tg),W(u,v)\ra_{\la\cS',\cS\ra}\\
		&-\la W(Tf,Tg),W(u,v)\ra_{\la\cS',\cS\ra}=0.
	\end{align*}
	It follows by Lemma \ref{lemmaDensity} that
	\[
		\la k_1-k_2,\Phi\ra_{\la\cS',\cS\ra}=0, \qquad \Phi\in\cS(\bR^{4d}),
	\]
	i.e., $k_1=k_2$.
\end{proof}


\begin{corollary} Under the assumptions of Theorem \ref{3.3},\\
	(1) $k\in L^2(\bR^{4d})$ if and only if $k_T\in L^2(\rdd)$.\\
	(2) $k\in\cS(\bR^{4d})$ if and only if $k_T\in\cS(\rdd)$.
	\end{corollary}
\begin{proof}
\textit{(1)} The equivalence  follows almost immediately by Moyal's identity:
$$ \|k_T\|^2_{L^2(\rdd)}=\|W k_T\|_{L^2(\bR^{4d})},$$
and a change of variables which gives
$$\|Wk_T\|_{L^2(\bR^{4d})}=\|\mathfrak{T}_pk\|_{L^2(\bR^{4d})}=\|k\|_{L^2(\bR^{4d})}.
$$
\textit{(2)} The equivalence $k_T\in \cS(\rdd)\Leftrightarrow  Wk_T \in \cS(\bR^{4d})$ follows  from the corresponding equivalence for the STFT, cf. \cite[Theorem 11.2.5]{book} or \cite[Theorem 1.2.23]{Elena-book}. Finally, observe that $\mathfrak{T}_p$ is a topological isomorphism on $\cS(\bR^{4d})$.
 \end{proof}

\section{Properties of FIO($S$)}\label{sec:Section4}
This part is devoted to illustrate the main properties of the class $FIO(S)$.
First, we show that $FIO(S)\subset B(L^2(\rd))$.

\begin{theorem}\label{Piold}
	An operator $T\in  FIO(S)$ is bounded on $\lrd$.
\end{theorem}
\begin{proof}
	For $g\in \lrd$ we recall \cite[Chapter 1]{Elena-book} that the Wigner $Wg\in\lrdd$ and by Moyal's identity:
$$\|Wg\|_{L^2(\rdd)}=\|g\|^2_{L^2(\rd)}.$$
	Using \eqref{I4}, for any $f\in\cS(\rd)$,
	$$\|Tf\|^2_{L^2(\rd)}=\|W(Tf)\|_{L^2(\rdd)} \leq \|K\|_{op}\|Wf\|_{L^2(\rdd)}=\|K\|_{op}\|f\|^2_{L^2(\rd)},
	$$
	since, by Definition \ref{C6T1.1old},  we have $k(z,w)=h(z,Sw)$, and $h$ is the kernel of a pseudodifferential operator with symbol in $S^0_{0,0}(\bR^{4d})$ which is bounded on   $L^2(\rdd)$ (see, e.g., \cite{A8,charly06}). This implies that $k$ is a kernel of an operator bounded on $\lrdd$, because the change of variables $w'=Sw$ gives 
	\begin{equation*}
		\|F(S\cdot)\|^2_{L^2(\rdd)}=\int_{\rdd}|F(Sw)|^2\,dw=\int_{\rdd}|F(w')|^2 dw'	=\|F\|^2_{L^2(\rdd)}
	\end{equation*}
(recall that $S$ is a symplectic matrix so that   $|\det S^{-1}|=1$). Since $\cS(\rd)$ is a dense subspace of $L^2(\rd)$ we infer $T\in B(L^2(\rd))$ with the norm estimate $\|T\|_{op}\lesssim \|K\|^{1/2}_{op}$. 
\end{proof}


Thus,  $\bigcup_{S\in Sp(d,\bR)}FIO(S)\subset B(\lrd)$. Let us show that the class  $\bigcup_{S\in Sp(d,\bR)}FIO(S)$ is a Wiener algebra. 
	\begin{theorem}\label{Pii}
	If $T_i\in FIO(S_i)$, $i=1,2$, then $T_1 T_2\in
	FIO(S_1 S_2)$.
\end{theorem}
\begin{proof}
	Using the Wigner representation in \eqref{I4} we can write 
	\begin{equation}\label{I4iie}
		W(T_1T_2f,T_1T_2g)(z) = \intrdd k_{1}(z,w) W(T_2f,T_2g)(w)\,dw,\quad f,g\in\cS(\rd),
	\end{equation}
	where $k_{1}(z,w)$ is the Wigner kernel of the operator $T_1$, satisfying \eqref{nucleoFIO} with symplectic matrix $S_1$. Similarly, 
	\begin{equation*}\label{I4ii}
		W(T_2f,T_2g)(w) = \intrdd k_{2}(w,u) W(f,g)(u)\,du,
	\end{equation*}
	with $k_{2}(w,u)$ being the Wigner kernel of $T_2$ satisfying \eqref{nucleoFIO} with symplectic transformation $S_2$. Substituting the expression of $W(T_2f,T_2g)(w)$ in \eqref{I4iie} we obtain
	\begin{equation}\label{e0old}
		W(T_1T_2f,T_1T_2g)(z) = \int_{\bR^{4d}} k_{1}(z,w)  k_{2}(w,u) W(f,g)(u)\,du\,dw,
	\end{equation}
	with, by  \eqref{nucleoFIO}, 
	\begin{equation}\label{e1}
		k_{i}(z,w)=h_{i}(z,S_iw),\quad z,w\in \rdd,\,\,i=1,2,
	\end{equation}
	and $h_{i}$ being the kernel of a pseudodifferential operator with symbol in $S^0_{0,0}(\bR^{4d})$.  Recall that  $W(f,g)\in\cS(\rdd)$.  By interpreting the integrals in the distributional sense, we can compute 
	\begin{equation}\label{e00}
		W(T_1T_2f,T_1T_2g)(z) = \int_{\bR^{2d}}\left(\intrdd k_{1}(z,w)  k_{2}(w,u)\,dw\right) W(f,g)(u)\,du.
	\end{equation}
In other words,  the Wigner kernel $k_{1 2}$ of the product $T_1T_2$ is given by
	$$k_{1 2}(z,u):= \intrdd k_{1}(z,w)  k_{2}(w,u)\,dw.
	$$
	We will  show that $k_{12}$ satisfies \eqref{nucleoFIO}.
The integral  kernel of an operator is connected with its Weyl symbol $\sigma$  by the formulae (cf. \cite[Formula (4.5)]{Elena-book}), that hold in $S'(\rdd)$:
$$\sigma= \cF_2\mathfrak{T}_L k\quad\mbox{or}\quad k=\mathfrak{T}_L^{-1}\cF_2^{-1}\sigma,$$
where the operator $\mathfrak{T}_L$ is defined in \eqref{A12prodotto}.
In our context, assuming that $\sigma_1,\sigma_2$ are the Weyl symbols related to the kernels 
$k_{1}, k_{2}$, respectively, we can write 
\begin{align}
		k_{12}(z,u)&:= \intrdd\mathfrak{T}_L^{-1}\cF_2^{-1}\sigma_1(z,S_1w)  \mathfrak{T}_L^{-1}\cF_2^{-1}\sigma_2(w,S_2u)\,dw\notag\\
		&= \intrdd\cF_2^{-1}\sigma_1(\frac{z+S_1w}{2},z-S_1w)  \cF_2^{-1}\sigma_2(\frac{w+S_2u}{2},w-S_2u)\,dw\notag\\
			&= \intrdd\cF_2^{-1}\sigma_1(\frac{z+w'}{2},z-w')  \cF_2^{-1}\sigma_2(\frac{S_1^{-1}w'+S_2u}{2},S_1^{-1}w'-S_2u)\,dw'\notag\\
			&= \intrdd\cF_2^{-1}\sigma_1(\frac{z+w'}{2},z-w')  \cF_2^{-1}\sigma_2(S_1^{-1}(\frac{w'+S_1S_2u}{2}),S_1^{-1}(w'-S_1S_2u))\,dw'\notag\\
			&= \intrdd\mathfrak{T}_L^{-1}\cF_2^{-1}\sigma_1(z,w')  \mathfrak{T}_L^{-1}\cF_2^{-1}\tilde{\sigma}_2(w',S_1S_2u)\,dw'\label{e7},		
\end{align}
where in the last-but-one row we made the change of variables $S_1w=w'$ hence $w=S_1^{-1}w'$ (recall $|\det S_1^{-1}|=1$) and in the last row we defined
$$\tilde{\sigma}_2(w,u):=\sigma_2(S_1^{-1}w,S_1^{-1}u).$$
It is straightforward to check that  $\sigma_2\in S^0_{0,0}(\bR^{4d})$ implies  $\sigma_2(S_1^{-1}\cdot,S_1^{-1}\cdot)\in S^0_{0,0}(\bR^{4d})$.
 Hence we can write
\begin{align*}
	k_{12}(z,u):= \intrdd k_{1}(z,w)  \tilde{k}_{2}(w,u)\,dw
	\end{align*}
with 	\begin{equation*}
\tilde{k}_{2}(z,w)=\tilde{h}_{2}(z,S_1S_2w),\quad z,w\in \rdd,
\end{equation*}
and 
$$\tilde{h}_{2}=\mathfrak{T}_L^{-1}\cF_2^{-1}\tilde{\sigma}_2,
$$
being the kernel of a pseudodifferential operator with symbol $\tilde{\sigma}_2$ in $S^0_{0,0}(\bR^{4d})$.  
The algebra's property of pseudodifferential operators with symbols in $S^0_{0,0}(\rdd)$ (cf. \cite{A8,charly06}) and the equality  in \eqref{e7} imply that the kernel of the product $k_{12}$ can be written as $$k_{12}(z,u)=h_{12}(S_1S_2z,u)=\mathfrak{T}_L^{-1}\cF_2^{-1}{\sigma}_{12}(z,S_1S_2u),$$ for a suitable Weyl symbol ${\sigma}_{12}\in S^0_{0,0}(\bR^{4d})$.  
\end{proof}


Finally, we prove the inverse-closedness of $\bigcup_{S\in Sp(d,\bR)} FIO(S)$, also called spectral invariance or Wiener's property, following the pioneering contributions \cite{A86,A87}.
We use the following preliminary result.
\begin{lemma}\label{elena} Assume $T\in FIO(S)$ invertible on $L^2(\rd)$, then the operator $K$ in \eqref{I3} is invertible on $\lrdd$ with inverse $K^{-1}$ satisfying
	\begin{equation}\label{I3inv}
		K^{-1} W(f,g)=W(T^{-1}f,T^{-1}g), \quad f,g\in\cS(\rd).
	\end{equation}
Namely, $K^{-1}$ is the operator with integral kernel given by the Wigner kernel of $T^{-1}$. 
	\end{lemma}
\begin{proof}
	Since $T\in FIO(S)$, the operator $K$ is bounded on $\lrdd$ (see the proof of Theorem \ref{Piold}). $T$  invertible implies $T^{-1}$ bounded on $\lrd$ so that the operator $K^{-1}$ in \eqref{I3inv} is well defined. Let us prove that $K^{-1}$ is bounded on $\lrdd$.
	By Moyal's identity,
	\begin{align*}
		\|K^{-1}W(f,g)\|_2^2&=\|W(T^{-1}f,T^{-1}g)\|_2^2=\|T^{-1}f\|_2^2\|T^{-1}g\|_2^2\\
		&\leq \|T^{-1}\|_{op}^4\|W(f,g)\|_2^2,
	\end{align*}
	for every $f,g\in\cS(\rd)$, so that $K^{-1}$ extends to a bounded operator on $\lrdd$ with $\|K^{-1}\|_{op}\leq \|T^{-1}\|_{op}^2$.
	
	 Let us show that 
	$KK^{-1}=K^{-1}K=I$, with $I$ the identity on $L^2(\rdd)$.
	By the density of the linear span of $W(f,g)$ in $\cS(\rdd)$ and the density of $\cS(\rdd)$ in $L^2(\rdd)$ it is enough to check the equalities above on $W(f,g)$, with $f,g\in\cS(\rd)$.
	Now, 
	$$ KK^{-1}W(f,g) =KW(T^{-1}f,T^{-1}g)=W(TT^{-1}f,TT^{-1}g)=W(f,g),$$
	for all $f,g\in\cS(\rd)$, so $KK^{-1}=I$. The equality $K^{-1}K=I$ is analogous. 
\end{proof}
\begin{theorem}
	Let $S\in Sp(d,\bR)$ and $T\in FIO(S)$ be invertible on $L^2(\rd)$. Then, $T^{-1}\in FIO(S^{-1})$.
\end{theorem}
\begin{proof}
	Let $T\in FIO(S)$ have Wigner kernel $k$ and let $h$ be the kernel of the pseudodifferential operator with symbol $\sigma\in S^0_{0,0}(\bR^{4d})$ of Definition \ref{C6T1.1old}. Let $\tilde k$ be the Wigner kernel of $T^\ast$. First, we prove that $T^\ast\in FIO(S^{-1})$ and $\tilde k(z,w)=\cI k(z,w)$, where $\cI k(u,v)=k(v,u)$. For, we use the symmetries of $W(\cdot)$, Theorem \ref{3.3} and the fact that the kernel of $T^\ast$ is $\cI\overline k_T$, where $k_T$ is the kernel of $T$.

	\begin{align*}
		 \tilde k(x,\xi,y,\eta)&=\mathfrak{T}_p W(\cI\bar k_T)(x,\xi,y,\eta)
		 =W(\cI\bar k_T)(x,y,\xi,-\eta)\\
		 &=W\bar k_T(y,x,-\eta,\xi)
		 =Wk_T(y,x,\eta,-\xi)\\
		&=\mathfrak{T}_pWk_T(y,\eta,x,\xi)
		=\mathcal{I}\mathfrak{T}_pWk_T(x,\xi,y,\eta),
	\end{align*}
	where $\mathfrak{T}_pF(x,\xi,y,\eta)=F(x,y,\xi,-\eta)$. Therefore,
	\[
		\tilde k = \mathcal{I}(\mathfrak{T}_pWk_T)=\cI k.
	\]

	Next, by (\ref{nucleoFIO}) we have
	\[
		\tilde k(z,w) = k(w,z)=h(w,Sz)=h (S(S^{-1}w),Sz).
	\]
	
		If the kernel $h$ has symbol in $S^0_{0,0}(\bR^{4d})$ then it is straightforward to show that the rescaled kernel  $h (S\cdot,S\cdot)$ has symbol in $S^0_{0,0}(\bR^{4d})$ as well.
	Defining $\tilde h(z,w):=h(Sw,Sz)$ we have 
	$$\tilde k(z,w)=\tilde h (z,S^{-1}w).$$
In other words,  $T^*\in FIO(S^{-1})$.

	Since $T$ is invertible on $L^2(\rd)$ (and so is $T^\ast$), the operator $P=T^\ast T$ is invertible on $L^2(\rd)$ and, by the algebra property (Theorem \ref{Pii}), $P\in FIO(I_{2d\times 2d})$. In other words, the operator $K_P$ related to $P$ in  \eqref{I3} is a pseudodifferential operator with symbol in $S^{0}_{0,0}(\bR^{4d})$. By Lemma \ref{elena}, $K_P$ is invertible and by  the Wiener's property of Gr\"ochenig and Rzeszotnik \cite[Theorem 6.3]{GR},  $K_P^{-1}$  is a pseudodifferential operator with symbol in $S^{0}_{0,0}(\bR^{4d})$ as well. This implies that the related operator $P^{-1}$ is in the class $FIO(I_{2d\times 2d})$. Applying again the algebra property of Theorem \ref{Pii}, we infer that $T^{-1}=P^{-1}T^\ast\in FIO(S^{-1})$,  as desired.
\end{proof}

As a byproduct we proved:
\begin{corollary}\label{aggiunto}
	If $T\in FIO(S)$ then its adjoint $T^*\in FIO(S^{-1})$.
\end{corollary}

\section{FIOs of type I}\label{sec:Section5}
This section is devoted to the study of the Wigner kernel of FIOs of type I. Namely, we consider  
\begin{equation}\label{tipo1}
	T_If(x)=\int_{\rd}e^{2\pi i\Phi(x,\xi)}\sigma(x,\xi)\hat f(\xi)d\xi, \qquad f\in\cS(\rd),
\end{equation}
in the general setting of the tame  phase $\Phi$ in Definition \ref{def2.1} and symbol $\sigma\in S^0_{0,0}(\rdd)$.

Our goal is to compute the Wigner kernel of $T_I$. First, we need some preliminary results.
\begin{proposition}\label{prop1}
Consider the FIO of type I in \eqref{tipo1} with symbol $\sigma\in S^0_{0,0}(\rdd)$ (resp. $\cS'(\rdd)$) and tame phase $\Phi$.
Then, for all $f,g\in\cS(\rd)$,
\[
	(T_If\otimes \bar g)(x,y)=T_1(f\otimes \bar g)(x,y),\quad x,y\in\rd,
\]
where
\begin{equation}\label{T1}
	T_1F(x,y)=\int_{\rdd}e^{2\pi i\Phi_1(x,y,\xi,\eta)}\sigma_1(x,y,\xi,\eta)\hat F(\xi,\eta)d\xi d\eta, \qquad F\in\cS(\rdd),
\end{equation}
with symbol $\sigma_1\in S^0_{0,0}(\bR^{4d})$ (resp. $\cS'(\bR^{4d})$):
\begin{equation*}
	\sigma_1(x,y,\xi,\eta):=\sigma(x,\xi)\otimes1(y,\eta),\quad x,y,\xi,\eta\in \rd,\\
\end{equation*}
and tame phase $\Phi_1$:
\begin{equation*}
	\Phi_1(x,y,\xi,\eta):=\Phi(x,\xi)+y\eta, \quad x,y,\xi,\eta\in \rd.
\end{equation*}
\end{proposition}
\begin{proof}
	It is  a straightforward application of the inversion formula in $\cS'(\rd)$. For $f,g\in\cS(\rd)$,
	\[
	\begin{split}
		(T_If\otimes\bar g)(x,y)&=\Big(\int_{\rd}e^{2\pi i\Phi(x,\xi)}\sigma(x,\xi)\hat f(\xi)d\xi\Big)\overline{g(y)}dy\\
	&=\Big(\int_{\rd}e^{2\pi i\Phi(x,\xi)}\sigma(x,\xi)\hat f(\xi)d\xi\Big)\Big(\int_{\rd}\hat{\bar g}(\eta)e^{2\pi i\eta y}d\eta\Big)\\
	&=\int_{\rdd}e^{2\pi i(\Phi(x,\xi)+y\eta)}\sigma(x,\xi)\hat f(\xi)\overline{\hat g(\eta)}d\xi d\eta.
	\end{split}
	\]
	 Note that $\Phi_1 \in\cC^{\infty}(\bR^{4d})$ and satisfies \eqref{phasedecay} and \eqref{detcond}, because
	\begin{equation}\label{Phi1}
		\partial^2_{x,y,\xi,\eta}\Phi_1(x,y,\xi,\eta)=\begin{pmatrix}
			\partial^2_{x,x}\Phi&0&	\partial^2_{x,\xi}\Phi&0\\
				0&0 &0&I\\
					\partial^2_{\xi,x}\Phi&0&	\partial^2_{\xi,\xi}\Phi&0\\
						0 &I&0&0\\
		\end{pmatrix}.
	\end{equation}
	Since $\sigma\in S^0_{0,0}(\rdd)$  (resp. $\cS'(\rdd)$) by assumption and the constant function $1(y,\eta)$ is trivially in $S^0_{0,0}(\rdd)$ (resp. $\cS'(\rdd)$), it follows that $\sigma_1\in S^0_{0,0}(\bR^{4d})$ (resp. $\cS'(\bR^{4d})$).
\end{proof}
\begin{corollary}\label{cor1}
	Under the assumptions of Proposition 	\ref{prop1}, if the symbol $\sigma\in S^0_{0,0}(\rdd)$ the operator $T_1:S(\rdd)\to S'(\rdd)$ extends to a continuous operator on $\lrdd$.
\end{corollary}
\begin{proof}
	Since $T_1$ is a FIO of type I with tame phase and symbol $\sigma_1\in S^0_{0,0}(\bR^{4d})$, the conclusion follows from Theorem 3.4 in \cite{CGNRJMPA}.
\end{proof}

Similarly, we can place the operator $T_I$ on the second factor $g$ of the tensor product, as illustrated below.
\begin{proposition}\label{prop2}
Under the assumptions of Proposition \ref{prop1}, for every $f,g\in\cS(\rd)$,
\[
	(f\otimes \overline{T_Ig})(x,y)=T_1'(f\otimes \bar g)(x,y),\quad x,y\in\rd,
\]
where $T_1'$ is the FIO of type I:
\begin{equation}\label{T1'}
	T_1'F(x,y)=\int_{\rdd}e^{2\pi i\Phi_1'(x,y,\xi,\eta)}\sigma_1'(x,y,\xi,\eta)\hat F(\xi,\eta)d\xi d\eta, \qquad F\in\cS(\rdd),
\end{equation}
with symbol $\sigma_1'\in S^0_{0,0}(\bR^{4d})$ (resp. $\cS'(\bR^{4d})$):
\begin{equation*}
\sigma_1'(x,y,\xi,\eta):=1(x,\xi)\otimes\overline{\sigma(y,-\eta)},\quad x,y,\xi,\eta\in\rd,
\end{equation*}
and tame phase  $\Phi_1'$:
	$$\Phi_1'(x,y,\xi,\eta):=\xi x-\Phi(y,-\eta),\quad x,y,\xi,\eta\in\rd.$$
\end{proposition}
\begin{proof}
	We use the same pattern as the one of Proposition \ref{prop1}. We write the details for sake of clarity. For $f,g\in\cS(\rd)$,
	\[
		\begin{split}
			(f\otimes \overline{T_Ig})(x,y)&=f(x)\overline{\int_{\rd}e^{2\pi i\Phi(y,\eta)}\sigma(y,\eta)\hat g(\eta)d\eta}\\
			&=\Big(\int_{\rd}\hat f(\xi)e^{2\pi i\xi x}d\xi\Big)\Big(\overline{\int_{\rd}e^{2\pi i\Phi(y,\eta)}\sigma(y,\eta)\hat g(\eta)d\eta}\Big)\\
			&=\int_{\rdd}e^{2\pi i(\xi x-\Phi(y,\eta))}\overline{\sigma(y,\eta)}\hat f(\xi)\overline{\hat g(\eta)}d\xi d\eta\\
			&=\int_{\rdd}e^{2\pi i(\xi x-\Phi(y,\eta))}\overline{\sigma(y,\eta)}\hat f(\xi)\hat{\overline{g}}(-\eta)d\xi d\eta\\
			&=\int_{\rdd}e^{2\pi i(\xi x-\Phi(y,-\eta))}\overline{\sigma(y,-\eta)}\hat f(\xi)\hat{\overline{g}}(\eta)d\xi d\eta.
		\end{split}
	\]
Finally, the same argument of the previous proposition gives $\sigma_1'\in S^0_{0,0}(\bR^{4d})$ if $\sigma\in S^0_{0,0}(\bR^{4d})$ and  $\Phi'_1$  tame. 
\end{proof}
\begin{corollary}
	Under the assumptions of Proposition 	\ref{prop2}, if the symbol $\sigma_1'\in S^0_{0,0}(\bR^{4d})$ the operator $T_1':S(\rdd)\to S'(\rdd)$ extends to a continuous operator on $\lrdd$.
\end{corollary}
\begin{proof}
It is the same argument as in Corollary \ref{cor1}.
\end{proof}

From the previous issues it is easy to infer the Wigner kernel of $T_I$, as demonstrated in the following propositions.
\begin{proposition}\label{prop12}
	If $\hat\cA\in Mp(2d,\bR)$ and $T_I$ is a FIO of type I in \eqref{tipo1} with $\sigma\in\cS'(\rdd)$, then
	\[
		W_\cA(T_If,g)=\hat\cA T_1\hat\cA^{-1}W_\cA(f,g),\quad f,g\in\cS(\rd),
	\]
	and
	\[
		W_\cA(f,T_Ig)=\hat\cA T_1'\hat\cA^{-1}W_\cA(f,g),
	\]
	where $T_1$ and $T_1'$ are defined in \eqref{T1} and \eqref{T1'}, respectively. As a consequence,
	\[
		W_\cA(T_If,T_Ig)=\hat\cA T_1T_1'\hat\cA^{-1}W_\cA(f,g), \qquad f,g\in\cS(\rd).
	\]
\end{proposition}
\begin{proof}	Using \eqref{T1},
	\[
		\begin{split}
			W_\cA(T_If,g)&=\hat\cA(T_If\otimes \bar g)=\hat\cA T_1(f\otimes\bar g)=\hat\cA T_1\hat\cA^{-1}W_\cA(f,g).
		\end{split}
	\]
	The second identity follows analogously and the third is a composition of the preceding two ones.
\end{proof}

\begin{lemma}\label{lemma12}
The product  $T_1T_1'$, where the operators $T_1$ and $T_1'$ are defined in \eqref{T1} and \eqref{T1'}, respectively, can be explicitly written as
\begin{equation}\label{T1T1p}
	T_1T_1'F(x,y)=\int_{\rdd}e^{2\pi i[\Phi(x,\omega)-\Phi(y,-\rho)]}\sigma(x,\omega)\overline{\sigma(y,-\rho)}\hat F(\omega,\rho)d\omega d\rho,\,\,F\in\cS(\rdd).
\end{equation}
\end{lemma}
\begin{proof}
	Take $F\in \cS(\rdd)$. Then,
	\begin{align*}
		T_1T_1'F(x,y)&=\int_{\rdd}e^{2\pi i[\Phi(x,\xi)+y\eta]}\sigma(x,\xi)\widehat{T_1'F}(\xi,\eta)d\xi d\eta\\
		&=\int_{\rdd}e^{2\pi i[\Phi(x,\xi)+y\eta]}\sigma(x,\xi)\int_{\rdd}T_1'F(u,v)e^{-2\pi i(u\xi+v\eta)}dudv d\xi d\eta\\
		&=\int_{\bR^{4d}}e^{2\pi i[\Phi(x,\xi)+y\eta-u\xi-v\eta]}\sigma(x,\xi)\int_{\rdd}e^{2\pi i[u\omega-\Phi(v,-\rho)]}\overline{\sigma(v,-\rho)}\\
		&\qquad\quad\times\hat F(\omega,\rho)d\omega d\rho dudvd\xi d\eta\\
		&=\int_{\bR^{6d}}e^{2\pi i[\Phi(x,\xi)-\Phi(v,-\rho)+y\eta-u\xi-v\eta+u\omega]}\\
		&\qquad\quad\times\sigma(x,\xi)\overline{\sigma(v,-\rho)}\hat F(\omega,\rho)dudvd\xi d\eta d\omega d\rho.
	\end{align*}
	Using the well-known formulae in $S'(\rd)$ (with the integrals meant in a formal way):
	\begin{align*}
		\int_{\rd}e^{-2\pi i(v-y)\eta}d\eta dv=\delta_y(v)dv\qquad \mbox{and}\qquad \int_{\rd}e^{-2\pi i(\xi-\omega)u}dud\xi=\delta_\omega(\xi)d\xi,
	\end{align*}
we can write (again, the integrals are formal)
	\begin{align*}
		T_1T_1'F(x,y)&=\int_{\bR^{4d}}e^{2\pi i[\Phi(x,\xi)-\Phi(v,-\rho)]}\sigma(x,\xi)\overline{\sigma(v,-\rho)}\hat F(\omega,\rho)\delta_y(v)dv\delta_\omega(\xi)d\xi d\omega d\rho\\
		&=\int_{\rdd}e^{2\pi i[\Phi(x,\omega)-\Phi(y,-\rho)]}\sigma(x,\omega)\overline{\sigma(y,-\rho)}\hat F(\omega,\rho)d\omega d\rho.
	\end{align*}
	This concludes the proof.
\end{proof}
\begin{corollary}
	If the symbol $\sigma\in S^0_{0,0}(\rdd)$, the product $T_1T_1'$ in \eqref{T1T1p} defines a type I FIO $T_1T_1':\cS(\rdd)\to\cS'(\rdd)$ which extends to a bounded operator on $\lrdd$.
\end{corollary}
\begin{proof}
It is straightforward to check that the phase $\tilde\Phi(x,y,\omega,\rho):=\Phi(x,\omega)-\Phi(y,-\rho)$ is a tame phase on $\bR^{4d}$ and that the symbol $\tilde\sigma(x,y,\omega,\rho):=\sigma(x,\omega)\overline{\sigma(y,-\rho)}$ is in $S^0_{0,0}(\bR^{4d})$ if $\sigma\in S^0_{0,0}(\bR^{2d})$. Then the conclusion follows  from Theorem 3.4 of \cite{CGNRJMPA}.
\end{proof}

Here we compute the Wigner kernel $k_I$ of the type I FIO in \eqref{tipo1}.
\begin{theorem}\label{teor3.6}
	Consider $T_I$ a FIO of type I defined in \eqref{tipo1} with symbol $\sigma\in S^0_{0,0}(\rdd)$. For $f\in\cS(\rd)$, 
	\begin{equation}\label{K}
		K(W(f,g))(x,\xi)=W(T_If,T_Ig)(x,\xi)=\int_{\rdd}k_I(x,\xi,y,\eta)W(f,g)(y,\eta)dyd\eta,
	\end{equation}
	where the Wigner kernel $k_I$ is
	\begin{equation}\label{KI}
			k_I(x,\xi,y,\eta)=\int_{\rdd}e^{2\pi i[\Phi_I(x,\eta,t,r)-(\xi t+ry)]} \sigma_I(x,\eta,t,r)dtdr,
	\end{equation}
with, for $x,\eta,t,r\in\rd$,
\begin{equation}\label{fii}
\Phi_I(x,\eta,t,r)=\Phi(x+\frac{t}{2},\eta+\frac{r}{2})-\Phi(x-\frac{t}{2},\eta-\frac{r}{2})
\end{equation}
and
\begin{equation}\label{sigmai}
\sigma_I(x,\eta,t,r):=\sigma(x+\frac{t}{2},\eta+\frac{r}{2})\overline{\sigma(x-\frac{t}{2},\eta-\frac{r}{2})}.
\end{equation}
\end{theorem}
\begin{proof}
	We use Proposition \ref{prop12} with $\cA=A_{1/2}$,  Lemma \ref{lemma12} and \eqref{rem1}. Observe that for $f,g\in\cS(\rd)$ the Wigner distribution $W(f,g)$ is in $\cS(\rdd)$,  metaplectic operators are continuous on $\cS(\rdd)$, so the operators' compositions below are well defined and we reckon 
	\[
		W(T_If,T_Ig)(x,\xi)=\hat A_{1/2}T_1T_1'\hat A_{1/2}^{-1}W(f,g)(x,\xi),\quad (x,\xi)\in\rdd,
	\]
so that
	\begin{align*}
		W(T_If,T_Ig)(x,\xi)&=\int_{\rd}(T_1T_1'\hat A_{1/2}^{-1}W(f,g))(x+t/2,x-t/2)e^{-2\pi it\xi}dt\\
		&=\int_{\rd}\int_{\rdd}e^{2\pi i[\Phi(x+\frac{t}{2},\omega)-\Phi(x-\frac{t}{2},-\rho)]}e^{-2\pi i\xi t}\\
		&\qquad\quad\times\,\,\sigma(x+\frac{t}{2},\omega)\overline{\sigma(x-\frac{t}{2},-\rho)}{(\hat A_{1/2}^{-1}W(f,g))^{\wedge}}(\omega,\rho)d\omega d\rho  dt\\
		&=\int_{\bR^{3d}}e^{2\pi i[\Phi(x+\frac{t}{2},\omega)-\Phi(x-\frac{t}{2},-\rho)-\xi t]}\sigma(x+\frac{t}{2},\omega)\overline{\sigma(x-\frac{t}{2},-\rho)}\\
		&\qquad\quad\times\,\,\int_{\rdd}\hat A_{1/2}^{-1}W(f,g)(u,v)e^{-2\pi i[u\omega +v\rho]}dudvd\omega d\rho dt\\
		&=\int_{\bR^{5d}}\underbrace{e^{2\pi i[\Phi(x+\frac{t}{2},\omega)-\Phi(x-\frac{t}{2},-\rho)-\xi t-u\omega-v\rho]}\sigma(x+\frac{t}{2},\omega)\overline{\sigma(x-\frac{t}{2},-\rho)}}_\text{$=:\phi(u,v,\omega,\rho,t)$}\\
		&\qquad\quad\times\,\,\hat A_{1/2}^{-1}W(f,g)(u,v)dudvd\omega d\rho dt\\
		&=\int_{\bR^{5d}}\phi(u,v,\omega,\rho,t)\int_{\rd}W(f,g)(u/2+v/2,\eta)e^{2\pi i(u-v)\eta}d\eta dudvd\omega d\rho dt.
	\end{align*}
	Using the change of variables $u/2+v/2=y$, 
	\begin{align*}
		W(T_If,T_Ig)(x,\xi)&=2^d\int_{\bR^{6d}}e^{2\pi i[\Phi(x+\frac{t}{2},\omega)-\Phi(x-\frac{t}{2},-\rho)-\xi t-(2y-v)\omega-v\rho-v\eta]}e^{2\pi i(2y-2v)\eta}\\
		&\qquad\quad\times\, \sigma(x+\frac{t}{2},\omega)\overline{\sigma(x-\frac{t}{2},-\rho)}W(f,g)(y,\eta)d\eta dydvd\omega d\rho dt.
	\end{align*}
In other words,
	\[
		W(T_If,T_Ig)(x,\xi)=\int_{\rdd}k_I(x,\xi,y,\eta)W(f,g)(y,\eta)dyd\eta,
	\]
	with
\begin{align*}
		k_I(x,\xi,y,\eta)&=2^d\int_{\bR^{4d}}e^{2\pi i[\Phi(x+\frac{t}{2},\omega)-\Phi(x-\frac{t}{2},-\rho)-\xi t-2y(\omega-\eta)+v(\omega-\rho-2\eta)]}\\
		&\qquad\quad\times\,\,\sigma(x+\frac{t}{2},\omega)\overline{\sigma(x-\frac{t}{2},-\rho)}dvd\omega d\rho dt.
			\end{align*}
Using the well-known formula in $\cS'(\rd)$ (with formal integral):
	\[
	\int_{\rd}e^{2\pi iv(\omega-\rho-2\eta)}dvd\omega=\delta_{2\eta+\rho}(\omega)d\omega,
	\]
we continue the computations as follows:
	\begin{align*}
	k_I(x,\xi,y,\eta)&=2^d\int_{\bR^{2d}}e^{2\pi i[\Phi(x+\frac{t}{2},2\eta+\rho)-\Phi(x-\frac{t}{2},-\rho)-\xi t-2y(\eta+\rho)]}\\
	&\qquad\quad\times \,\, \sigma(x+\frac{t}{2},2\eta+\rho)\overline{\sigma(x-\frac{t}{2},-\rho)}d\rho dt\\
	&=2^d\int_{\bR^{2d}}e^{2\pi i[\Phi(x+\frac{t}{2},2\eta+\rho)-\Phi(x-\frac{t}{2},-\rho)-\xi t-2y\eta-2y\rho]}\\
	&\qquad\quad\times \,\,\sigma(x+\frac{t}{2},\rho)\overline{\sigma(x-\frac{t}{2},-\rho)}d\rho dt\\
	&=2^d e^{-4\pi iy\eta}\int_{\bR^{2d}}e^{2\pi i[\Phi(x+\frac{t}{2},2\eta+\rho)-\Phi(x-\frac{t}{2},-\rho)-\xi t-2y\rho]}\\
	&\qquad\qquad\qquad\times \,\,\sigma(x+\frac{t}{2},2\eta+\rho)\overline{\sigma(x-\frac{t}{2},-\rho)}d\rho dt\\
	&=e^{-4\pi iy\eta}\int_{\rdd}e^{2\pi i[\Phi(x+\frac{t}{2},2\eta+\frac{\rho}{2})-\Phi(x-\frac{t}{2},-\frac{\rho}{2})]}e^{-2\pi i(\xi t+y\rho)}\\
		&\qquad\qquad\quad\times\,\,\sigma(x+\frac{t}{2},2\eta+\frac{\rho}{2})\overline{\sigma(x-\frac{t}{2},-\frac{\rho}{2})}d\rho dt.
	\end{align*}
Finally, the change of variables  $$\eta+\rho/2=r/2$$
 (so that \, $2\eta+\rho/2=\eta+r/2$ and  $-\rho/2=\eta-r/2$) gives the Wigner kernel \eqref{KI}.
\end{proof}
\subsection{Properties of the Wigner kernel}
In this subsection we study the properties of the Wigner kernel $k_I$. Let us first study the general case of a FIO $T_I$  with tame phase $\Phi$, in view of future applications.
\begin{theorem}\label{T4.1}
	Let $T_I$ be a FIO of type I as defined in \eqref{tipo1} with tame phase (Definition \ref{def2.1}). Assume its symbol $\sigma$ is in $\cS(\rdd)$. Then its Wigner kernel $k_I$ can be expressed as
		\begin{equation}\label{KI-exp}
		k_I(x,\xi,y,\eta)=\int_{\rdd} e^{-2\pi i[t\cdot(\xi-\Phi_x(x,\eta))+ r\cdot(y-\Phi_\eta(x,\eta))]}\tilde{\sigma}(x,\eta,t,r)\,dtdr
	\end{equation}
where
\begin{equation}\label{sigmatilde}
	\tilde{\sigma}(x,\eta,t,r)=e^{2\pi
		i[\Phi_{2}-\widetilde{\Phi}_{2}](x,\eta,t,r)} \sigma(x+\frac{t}{2},\eta+\frac{r}{2})\overline{\sigma(x-\frac{t}{2},\eta-\frac{r}{2})}
\end{equation}

with \begin{equation}
	\label{eq:c11}
	\Phi_{2}(x,\eta,t,r)=\sum_{|\a|=2}\int_0^1(1-\tau)\partial^\a
	\Phi(\phas+\tau(t,r)/2)\,d\tau\frac{(t,r)^\a}{2^3\a!}.
\end{equation}
and 
\begin{equation}
	\label{eq:c12}
	\widetilde{\Phi}_{2}(x,\eta,t,r)=\sum_{|\a|=2}\int_0^1(1-\tau)\partial^\a
	\Phi(\phas-\tau(t,r)/2)\,d\tau\frac{(t,r)^\a}{2^3\a!}.
\end{equation}
Moreover, 
	\begin{equation}\label{E0}
		k_I(x,\xi,y,\eta)=\frac1{\la
			2\pi(y-\Phi_\eta(x,\eta),\xi-\Phi_x(x,\eta))\ra^{2N}}\cF_2 a_N(x,\eta,\xi-\Phi_x(x,\eta),y-\Phi_\eta(x,\eta)),
	\end{equation}
where, for $u=(t,r)$, 
	\begin{equation}\label{KN}
		a_N(x,\eta,t,r)=(1-\Delta_u)^N\tilde{\sigma}(x,\eta,t,r).
	\end{equation}
	
\end{theorem}
\begin{proof}
	Since $\Phi$ is smooth, we can expand $\Phi(x+\frac{t}{2},\eta+\frac{r}{2})$ and $\Phi(x-\frac{t}{2},\eta-\frac{r}{2})$
	into a Taylor series around
	$\phas$. Namely,
	\begin{equation}\label{E1}
		\Phi\left(x+\frac{t}{2},\eta+\frac{r}{2}\right)=\Phi\phas+\frac {t}2\Phi_x \phas+\frac {r}2 \Phi_\eta \phas +\Phi_{2}(x,\eta,t,r),
	\end{equation}
	where the remainder $\Phi_{2}$ is given by \eqref{eq:c11}.
	Similarly, 
	\begin{equation}\label{E2}
		\Phi\left(x-\frac{t}{2},\eta-\frac{r}{2}\right)=\Phi\phas-\frac {t}2\Phi_x \phas-\frac {r}2 \Phi_\eta \phas +\widetilde{\Phi}_{2}(x,\eta,t,r),
	\end{equation}
	with
	$\widetilde{\Phi}_{2}$ defined in \eqref{eq:c12} above.
	Inserting the phase expansions above in \eqref{KI} we obtain \eqref{KI-exp}, 
	where
	$\tilde{\sigma}$ is defined in \eqref{sigmatilde}.
	
	For $N\in\bN$, $u=(t,r)\in\rdd$, using the identity:
	\begin{multline*}(1-\Delta_u)^Ne^{-2\pi i[(\xi-\Phi_x(x,\eta), y-\Phi_\eta(x,\eta))\cdot (t,r)]}\\
		=\la
		2\pi(\xi-\Phi_x(x,\eta), y-\Phi_\eta(x,\eta))\ra^{2N} e^{-2\pi i [(\xi-\Phi_x(x,\eta), y-\Phi_\eta(x,\eta))\cdot (t,r)]},
	\end{multline*}
	we integrate by parts in \eqref{KI-exp} and obtain
	\begin{align*}
		k_I(x,\xi,y,\eta)&=\frac1{\la
			2\pi(\xi-\Phi_x(x,\eta), y-\Phi_\eta(x,\eta))\ra^{2N}}\int_{\rdd} e^{-2\pi i [(\xi-\Phi_x(x,\eta), y-\Phi_\eta(x,\eta))\cdot (t,r)]} \\
		&\qquad\qquad\times (1-\Delta_u)^N \tilde{\sigma}(x,\eta,t,r)\,dtdr\\
		&=\frac1{\la
			2\pi(\xi-\Phi_x(x,\eta), y-\Phi_\eta(x,\eta))\ra^{2N}}\int_{\rdd} e^{-2\pi i [(\xi-\Phi_x(x,\eta), y-\Phi_\eta(x,\eta))\cdot (t,r)]} \\
		&\qquad\qquad\times a_N(x,\eta,t,r)\,dtdr,
	\end{align*}
	with the symbol $a_N$ defined in \eqref{KN}. This gives the claim \eqref{E0}.
\end{proof}

We now focus on the type I FIO with canonical transformation $\chi=S\in Sp(d,\bR)$. In this case the expression of the kernel $k_I$ in \eqref{KI-exp} simplifies drastically. 
\begin{theorem}\label{stime}
	Assume a linear symplectic transformation $S\in Sp(d,\bR)$ with block-decomposition in \eqref{S} and related phase $\Phi$ in \eqref{fase}, and consider $T_I$ the FIO of type I in \eqref{tipo1} having  symbol $\sigma\in S^0_{0,0}(\rdd)$. Then the Wigner kernel is given by
		\begin{equation}\label{klS}
		k_I(x,\xi,y,\eta)=\int_{\rdd} e^{-2\pi i[t\cdot(\xi-\Phi_x(x,\eta))+ r\cdot(y-\Phi_\eta(x,\eta))]}\tilde{\sigma}(x,\eta,t,r)\,dtdr
	\end{equation}
	with 
	\begin{equation}\label{tildesigma}
		\tilde{\sigma}(x,\eta,t,r)=\sigma(x+\frac t2,\eta+\frac r2)\overline{\sigma(x-\frac t2,\eta-\frac r2)}\in S^0_{0,0}(\bR^{4d}).
	\end{equation}
	Moreover, $T_I\in FIO(S)$.
\end{theorem}
\begin{proof}
	We use the expression of the kernel in formula \eqref{KI} and compute the argument of the exponential. Here the phase $\Phi$ is explicitly given by \eqref{fase}.
	An easy computation yields
	$$\Phi(x+\frac{t}{2},\eta+\frac{r}{2})-\Phi(x-\frac{t}{2},\eta-\frac{r}{2})=xCA^{-1}t+\eta A^{-1}
	t+rA^{-1}x-\eta A^{-1}B r
	$$
	(recall that $CA^{-1}$ and $A^{-1}B$ are symmetric matrices and $D=CA^{-1}B+A^{-T}$, cf.  \cite{Elena-book,Gos11}).
	Hence we can write
\begin{align*}
		 xCA^{-1}t+\eta A^{-1}
	t+rA^{-1}x-\eta A^{-1}B r&=t\cdot(CA^{-1}x+ A^{-T}\eta) +r\cdot(A^{-1}x-A^{-1}B \eta)\\
	&=t\cdot\Phi_x(x,\eta)+r\cdot \Phi_\eta(x,\eta),
\end{align*}
where
\begin{equation}\label{gradiente fasi}
	\Phi_x(x,\eta)=CA^{-1}x+ A^{-T}\eta\quad \mbox{and}\quad\Phi_\eta(x,\eta)= A^{-1}x-A^{-1}B \eta
\end{equation}
	so that 
	\begin{equation}\label{midStepGG}
			k_I(x,\xi,y,\eta)=\int_{\rdd} e^{-2\pi i[t\cdot(\xi-\Phi_x(x,\eta))+ r\cdot(y-\Phi_\eta(x,\eta))]}\tilde{\sigma}(x,\eta,t,r)\,dtdr,
	\end{equation}
i.e., we obtain formula \eqref{klS} with kernel 
$\tilde{\sigma}$ in \eqref{tildesigma} in $S^0_{0,0}(\bR^{4d})$, 
since $\sigma\in S^0_{0,0}(\rdd)$.

We now prove that  $T_I\in FIO(S)$. Using the symplectic relations for $S$ and the relations $CA^{-1},A^{-1}B\in\Sym(d,\bR)$ and $D=CA^{-1}B+A^{-T}$, we can rewrite (\ref{midStepGG}) as:
\begin{align*}
	k_I(x,\xi,y,\eta)&=\cF_2\tilde	\sigma(x,\eta,\xi-\Phi_x(x,\eta),y-\Phi_\eta(x,\eta))\\
	&=\cF_2\tilde\sigma(x,\eta,\xi-CA^{-1}x-A^{-T}\eta,y-A^{-1}x+A^{-1}B\eta)\\
	&=\cF_2\tilde\sigma(R_I(x,\xi,S(y,\eta)-(x,\xi)))\\
	&=\mathfrak{T}_{R_I}\cF_2\tilde\sigma(x,\xi,S(y,\eta)-(x,\xi)),
\end{align*}
where $\mathfrak{T}_{R_I}F=F\circ R_I$, and
	\begin{equation}\label{defdellaM}
	R_I=\left(\begin{array}{cc|cc}
		I & 0 & 0 & 0\\
		-C^T & A^T & -C^T & A^T\\
		\hline
		0 & 0& CA^{-1} & -I\\
		0 & 0 & A^{-1} & 0
	\end{array}\right).
\end{equation}
Since $R_I$ is upper triangular, by Proposition \ref{propA1} we have that for a suitable $\hat U_I\in Mp(4d,\bR)$ having upper triangular projection $U_I\in Sp(4d,\bR)$,
\begin{align*}
	k_I(x,\xi,y,\eta)&=\mathfrak{T}_{R_I}\cF_2\tilde\sigma(x,\xi,S(y,\eta)-(x,\xi))\\
	&=\cF_2(\hat U_I\tilde\sigma)(x,\xi, S(y,\eta)-(x,\xi)),
\end{align*}
we refer to Appendix \ref{A1} for the proof of this intertwining relation, which can be checked directly trough matrix multiplication. Since $\tilde\sigma\in S^0_{0,0}(\bR^{4d})$, we have that $\tilde\sigma\in M^{\infty,1}_{1\otimes v_s}(\bR^{4d})$ for every $s\geq0$, cf. \cite[Chapter 2]{Elena-book}. By \cite[Theorem 4.6]{Fuhr}, $\hat U_I\tilde\sigma\in M^{\infty,1}_{1\otimes v_s}(\bR^{4d})$ for every $s\geq0$ and, therefore, $\hat U_I\tilde\sigma\in S^0_{0,0}(\bR^{4d})$ by \cite[Lemma 2.2]{BC2021}. Therefore, setting $\hat U_I\tilde\sigma=\tilde{\tilde\sigma}\in S^0_{0,0}(\bR^{4d})$,
\[
	h(x,\xi,S(y,\eta))=k_I(x,\xi,y,\eta)=\cF_2\tilde{\tilde\sigma}(x,\xi,S(y,\eta)-(x,\xi)),
\]
i.e.
\[
	h(x,\xi,y,\eta)=\cF_2\tilde{\tilde\sigma}(x,\xi,(y,\eta)-(x,\xi)),
\]
that is $h$ is the kernel of a pseudodifferential operator with Kohn-Nirenberg symbol $\tilde{\tilde\sigma}\in S^0_{0,0}(\bR^{4d})$.

\end{proof}

\begin{remark}\label{e8}
For a quadratic phase $\Phi$, Theorem \ref{T4.1} simplifies considerably, with $\tilde\sigma$ made of a product of two linear transformations of the original symbol $\sigma$, without any reminder term of the phase $\Phi$. For $u=(t,r)$, using
\begin{multline*}(1-\Delta_u)^Ne^{-2\pi i[(\xi-\Phi_x(x,\eta), y-\Phi_\eta(x,\eta))\cdot (t,r)]}\\
	=\la
	2\pi(\xi-\Phi_x(x,\eta), y-\Phi_\eta(x,\eta))\ra^{2N} e^{-2\pi i [(\xi-\Phi_x(x,\eta), y-\Phi_\eta(x,\eta))\cdot (t,r)]},
\end{multline*}
we integrate by parts in \eqref{klS} and use Lemma 6.1.4. in \cite{Elena-book} to infer, for every $N\in\bN$,
\begin{align*}
	k_I(x,\xi,y,\eta)&=\frac1{\la
		2\pi(\xi-\Phi_x(x,\eta), y-\Phi_\eta(x,\eta))\ra^{2N}}\int_{\rdd} e^{-2\pi i [(\xi-\Phi_x(x,\eta), y-\Phi_\eta(x,\eta))\cdot (t,r)]} \\
	&\qquad\quad\times\quad(1-\Delta_u)^N\tilde{\sigma}(x,\eta,t,r)\,dtdr
\end{align*}
that we can estimate by
\begin{align*}
\quad	&\frac1{\la
		2\pi(\xi-S_2(y,\eta), x-S_1(y,\eta))\ra^{2N}}\int_{\rdd} e^{-2\pi i [(\xi-\Phi_x(x,\eta), y-\Phi_\eta(x,\eta))\cdot (t,r)]} \\
	&\qquad\quad\times\quad(1-\Delta_u)^N\tilde{\sigma}(x,\eta,t,r)\,dtdr\\
	&=\frac1{\la
		2\pi(\xi-S_2(y,\eta), x-S_1(y,\eta))\ra^{2N}}h_N(x,\xi,S(y,\eta))
\end{align*}
where, arguing as in the proof above,  $h_N$ is the kernel of a pseudodifferential operator with symbol in $ S^0_{0,0}(\bR^{4d})$. In particular, if $T$ is a Kohn-Nirenberg pseudodifferential operator, the phase is $\Phi(x,\eta)=x\eta$, so that $\Phi_\eta(x,\eta)=x$ and $\Phi_x(x,\eta)=\eta$,
and the related symplectic map is the identity $I_{2d\times2d}$. Hence, for every $N\in\bN$, if we set $z=(x,\xi)$, $w=(y,\eta)$, then
$$	k_I(z,w)=\frac1{\la
	z-w\ra^{2N}}h_N(z,w),$$
with $h_N$ kernel of a pseudodifferential operator with symbol in $\cS^0_{0,0}(\bR^{4d}).$
  Thus, we recover Theorem 1.4 in \cite{CRPartI2022}.
\end{remark}

We end up this section with a few comments related to the $L^2$-adjoint of a FIO of type $I$, called FIO of type II, written  formally as
\begin{equation}\label{FIOII}
	T_{II}f(x)=\int_{\rdd}e^{-2\pi i[\Phi(y,\xi)-x\xi]}\tau(y,\xi)f(y)dyd\xi, \qquad f\in\cS(\rd).
\end{equation}
As for the type I case, we consider  symbols $\tau$ in the H\"ormander class $S^0_{0,0}(\rdd)$ and tame phase functions $\Phi$, cf. Definition \ref{def2.1}. These assumptions yield that $T_{II}$ is bounded on $\lrd$ and  $T_{II}=T_I^{\ast}$, see for
example \cite{AS78}.\par
From Corollary \ref{aggiunto}, if $T_I \in FIO(S)$ then  $T_{II}\in FIO(S^{-1})$.

\section{Application to Schr\"{o}dinger equations}\label{sec:Section7}
The theory developed so far can be successfully applied in the context of Quantum Mechanics. 
Precisely, we can represent the Wigner kernel of the  Schr\"{o}dinger propagator $e^{itH}$,  solution to the Cauchy problem \eqref{C6intro} with Hamiltonian $H=a(x,D)+\sigma(x,D)$ in \eqref{PropH}. We consider $a(x,D)$,   the quantization of a quadratic form, and the perturbation  $\sigma(x,D)$ with symbol $\sigma\in S^0_{0,0}(\rdd)$, as explained in the introduction.
\begin{theorem}\label{teofinal}
	Let $H$ be the Hamiltonian in \eqref{PropH} with
quadratic polynomial $a$ and $\sigma\in
	S^0_{0,0}(\rdd)$. Let $e^{itH}$ be the corresponding
	propagator. Then $e^{itH}\in FIO(S_t)$ for every $t \in\bR$, where $S_t\in Sp(d,\bR)$ is the Hamiltonian flow defined by \eqref{eitA}, $S_t(y,\eta)=(x(t,y,\eta), \xi(t,y,\eta))$. 
	Namely, the Wigner kernel for the solution of the homogenous problem $i\frac{\partial}{\partial t} u + Op_w(a) u = 0$ is given by 
	\[
k(t,z,w) =h_t(z,S_tw)
	\]
where $h_t$ is the kernel of a pseudodifferential operator of symbol  $b_t\in  S^0_{0,0}(\bR^{4d})$, with continuous dependence of $t\in\bR$.
\end{theorem}
\begin{proof}
Key tools are the results and pattern of \cite{helffer84}. Precisely,   there exists a $T>0$ such that  $e^{itH}$  can be represented as a type I FIO as in \eqref{C61.13}. By Theorem \ref{stime}, we infer $$e^{itH}\in FIO(S_t),\quad t\in (-T,T),$$
where $S_t\in Sp(d,\bR)$ is the canonical transformation satisfying  condition \eqref{detcond2} and related to the tame phase $\Phi(t,x,\xi)$ in formula  \eqref{C61.13} by  \eqref{cantra}. In other words, the Wigner kernel $k(t,z,w)$  of $e^{itH}$  satisfies the equality \eqref{nucleoFIOt} for $t\in(-T,T)$.

Observe that in general the product of two type I FIOs is not a type I FIO as well, so we can not use an approximation argument to get \eqref{nucleoFIOt} for every $t$ working directly on $e^{itH}$.
Instead,  we use the algebra property of $FIO(S_t)$ and the classical trick we report below for sake of clarity.

Consider $T_0<T/2$ and define $I_m=(mT_0,(m+2)T_0)$, $m\in\bZ$. For $t\in\bR$, there exists an $m\in\bZ$ such that $t \in I_m$. For $t_1=t-mT_0\in (-T,T)$ we have that $e^{it_1H}\in FIO(S_{t_1})$ and for $t_2=\frac{m}{|m|}T_0\in (-T,T)$, $e^{it_2H}\in FIO(S_{t_2})$.
Using the semigroup property of $e^{itH}$ and the  algebra property in Theorem \ref{tc1old} $(ii)$ we obtain
$$e^{it_1H}(e^{it_2H})^{|m|}\in FIO(S_{t_1}S_{t_2}^{|m|}).$$
By the group property of $S_t$:
$$S_{t_1}S_{t_2}^{|m|}=S_{t_1+|m|t_2}=S_t,$$
so that
\begin{equation}\label{meta}
	e^{itH}\in FIO(S_t),\quad t\in \bR,
\end{equation}
and the kernel representation in \eqref{nucleoFIOt} holds for every $t\in\bR$. Observe also the continuity with respect to $t$ in the construction of $h_t$.
\end{proof}

Let us stress that this representation  holds for every $t\in\bR$, also in the caustic points, where \eqref{C61.13} fails.
\section{Examples}\label{7}
Let us clarify the statement of Theorem \ref{teofinal} with some examples other than the Harmonic Oscillator presented in the Introduction. Consider the case when $\sigma=0$, i.e., the Hamiltonian is a quadratic polynomial. Wigner kernels for this case were already considered in \cite{CGRPartII2022} in the lines of Wigner's original work \cite{Wigner}.
Let us now study the Wigner kernel under our new perspective.  

A global solution of 
\begin{equation*}
	\begin{cases} i \displaystyle\frac{\partial
			u}{\partial t} +a(x,D)u=0,\\
		u(0,x)=u_0(x).
	\end{cases}
\end{equation*}
can be expressed as a metaplectic operator (cf.  de Gosson \cite{Gos11}).

In general, the solution can be expressed by a FIO of type I only locally in time because of the caustics. Instead, the Wigner kernel has the global expression 
$$k(t,z,w)=\intrdd e^{2\pi i (z-S_tw)y }\,dy=\delta_{z-S_tw},$$
where $S_t$ is the symplectic Hamiltonian flow of $a(x,D)$.
So in this case the pseudodifferential operator $h=h_t$ in Theorem \ref{teofinal} is the identity, with kernel $\delta_{z-w}$.
Let us detail two relevant cases.\\
{\bf The Free Particle.} Here $a(x,D)=\Delta$ and a FIO of type I defines a global solution:
$$T_t f(x)=\intrd e^{2\pi i( x\xi- t
	\xi^2)} {\widehat {f}}(\xi)d\xi,\quad t\in\bR,\,\, x\in\rd.
$$
We have 
$$W(T_t f) (x,\xi)=Wf(x-t\xi,\xi)$$
with Wigner kernel 
\begin{equation*}
	k=\delta_{z-S_tw},\quad w,z\in\rdd,
\end{equation*}
where, if we write $w=(y,\eta)$, then $S_t(y,\eta)=(y+t\eta,\eta)$.\\
\noindent
{\bf Uniform magnetic potential.}  	This example was inspired by \cite[Section 4]{Helge}. 
 	Consider the uniform magnetic potential 
 	\[
 	A : \mathbb{R}^d \to \mathbb{R}^d, \quad x \mapsto m\omega Bx,
 	\]
 	where the mass $m>0$, and \( B \) is a real-valued \( d \times d \) matrix such that \( B^\top = -B \) and \( B^\top B = I \). For instance, in dimension  \( d = 2 \), the potential is of the form 
 	$
 	A(x_1, x_2) = \pm m\omega(-x_2, x_1).
 	$
 	In general, the operator related to the uniform magnetic potential is 
 	\[ a(x,D)= 	\frac{1}{2m} \left( \nabla - iA(x) \right)^2
 	\]
 	and the associated Hamiltonian is given by
 	\[
 	H(z) = \frac{1}{2m} \left\| \xi + A(x) \right\|^2 
 	= \frac{1}{2m} \xi^2 + \omega \left(  Bx\cdot \xi -  B\xi\cdot x \right) + \frac{m\omega^2}{2} x^2,
 	\]
 	where \( z = (x, \xi) \).
 	The Hamiltonian can be expressed as the inner product \( H(z) = \frac{1}{2} \langle Mz, z \rangle \), with 
 	\[
 	M = 
 	\begin{pmatrix}
 		m\omega^2 I & -\omega B \\
 		\omega B & \frac{1}{m} I
 	\end{pmatrix},
 	\]
 	so that 
 	\[
 	X := JM = 
 	\begin{pmatrix}
 		\omega B & \frac{1}{m} I \\
 		-m\omega^2 I & \omega B
 	\end{pmatrix} 
 	\]
 	(where $J$ is the standard symplectic matrix in \eqref{J}) is the matrix in the symplectic algebra $ \mathfrak{sp}(4, \mathbb{R})$. Note that there is an error in the computation of $M$ in \cite{Helge} which propagates throughout $X$ and $S_t$. 
 	
 		The canonical transformation $S_t$ is then obtained by computing 
 		\begin{equation}\label{theSummation}
 			S_t=	\exp(tX) = \sum_{k=0}^\infty \frac{t^k}{k!} X^k.
 		\end{equation}
 		Using that $B^2=-I$, we reckon
 		\[
 		X^2  = 2\omega
 		\begin{pmatrix}
 			B& 0 \\
 			0 & B
 		\end{pmatrix}X.
 		\]
 	A standard induction argument allows to distinguish between the matrices with even and odd exponents. In fact, using again that \( B^2 = -I \), we find
 	\[
 	X^{2k} = (-1)^{k-1} (2\omega)^{2k-1} \begin{pmatrix}
 		B& 0 \\
 		0 & B
 	\end{pmatrix}X\quad \text{and} \quad X^{2k+1} = (-1)^{k} (2\omega)^{2k} X \quad \text{for } k > 0.
 	\]
 	Thus, the summation \eqref{theSummation} becomes
 	\begin{align*}
 	S_t&=\exp(tX) =
 	I+\sum_{k=1}^\infty 
 	\frac{(-1)^{k-1} (2\omega)^{2k} t^{2k}}{(2k)!}  \begin{pmatrix}
 		B& 0 \\
 		0 & B
 	\end{pmatrix}X
 	+
 	\sum_{k=0}^\infty 
 	\frac{ (-1)^{k} (2\omega)^{2k}t^{2k+1}}{(2k+1)!}   X\\
 	&=I+\left(-\sum_{k=0}^\infty 	\frac{(-1)^{k} (2\omega)^{2k} t^{2k}}{(2\omega)(2k)!}+\frac{1}{2\omega}\right)\begin{pmatrix}
 		B& 0 \\
 		0 & B
 	\end{pmatrix}X + \sum_{k=0}^\infty 
 	\frac{ (-1)^{k} (2\omega)^{2k+1}t^{2k+1}}{(2\omega)(2k+1)!}   X\\
 	&=	I+\frac{1}{2\omega}\left( 1-\cos(2\omega t)\right)	\begin{pmatrix}
 		B& 0 \\
 		0 & B
 	\end{pmatrix}X+\frac{1}{2\omega}\sin(2\omega t)X\\
&=\frac12\begin{pmatrix}
\left(1+\cos(2\omega t)\right)I+\sin(2\omega t)B &	\frac{1}{m\omega}\left[\left(1-\cos(2\omega t)\right)B+\sin(2\omega t)I \right]\\
-{m\omega}\left[\left(1-\cos(2\omega t)\right)B+\sin(2\omega t)I\right] & \left(1+\cos(2\omega t)\right)I+\sin(2\omega t)B
\end{pmatrix}\\
=:&\begin{pmatrix}
	\mathcal{A}(t)& \mathcal{B}(t) \\
	\mathcal{C}(t)&\mathcal{D}(t)
\end{pmatrix}.
 	\end{align*}
 Using $B^T=-B$ easy computations yield
 $$\mathcal{A}(t)^T\mathcal{C}(t)=-\frac{m\omega}{2}\sin(2\omega t)I.$$  
  Hence, for \begin{equation*}
  	t\not=\frac{k\pi}{2\omega},
  \end{equation*}
the product and each factor $\mathcal{A}(t)$ and $\mathcal{C}(t)$ is invertible. In particular, for those $t$ so that $\mathcal{A}(t)$ is invertible, we retrieve the type I FIO representation (see \eqref{f4} and \eqref{fase}):
	\[
	u(t,x) = (\det\mathcal{A}(t))^{-1/2}
	\int
	e^{- \pi i x \cdot \mathcal{C}(t)\mathcal{A}(t)^{-1}x + 2\pi i \xi \cdot \mathcal{A}(t)^{-1}x + \pi i \xi \cdot \mathcal{A}(t)^{-1}\mathcal{B}(t) \xi} \, \hat{u}_0(\xi) \, d\xi.
	\]
In the points 
\begin{equation*}
	t_k:=\frac{k\pi}{2\omega},\quad k\in\bZ,
\end{equation*}
at least one among $\cA(t)$ and $\cC(t)$ is singular. In these cases, the simplectic matrix $S_{t_k}$ becomes
\[S_{t_k}=\frac12\begin{pmatrix}
	(1+(-1)^k)I& \frac{1}{m\omega}(1-(-1)^k)B\\
	-m\omega(1-(-1)^k)B&(1+(-1)^k)I
\end{pmatrix}.\]
For $k=2n$, $n\in\bZ$, we obtain $S_{t_{2n}}=I$, the identity matrix, and consequently
$$u(t,x)=\int_{\rd} e^{2\pi i x\xi} \hat{u}_0(\xi)\,d\xi.$$
On the other hand, the points 
\begin{equation*}
	t_{2n+1}=\frac{(2n+1)\pi}{2\omega},\quad n\in\bZ,
\end{equation*}
corresponding to the canonical transformation
\[S_{t_{2n+1}}=\frac12\begin{pmatrix}
	0& \frac{2}{m\omega}B\\
	-2m\omega B&0
\end{pmatrix},\]
are the \emph{caustics}. In these points, the FIO of type I representation of $u(t,x)$ fails, whereas the Wigner kernel 
\begin{equation*}
	k(t,z,w)=\delta_{z-S_t w},\quad w,z\in\rdd, \,t\in\bR,
\end{equation*}
is a well-defined tempered distribution. 
\section{Comments and additional references}\label{sec:Section8}
Our study of the Wigner kernel \eqref{nucleoFIOt} can be seen as part of a general project in microlocal analysis. Roughly speaking, let us fix a microlocal representation, i.e., a simultaneous description of energy in the space and frequency variables $(x,\xi)\in\rdd$, for a function or distribution $f$ in $\rd$. Given  a linear operator $T$, typically a pseudodifferential operator or Fourier integral operator, one aims to establish a direct relation $K$ between the microlocal representation of $f$ and that of $Tf$. The final objective is to infer relevant hidden properties of $T$ by $K$. 

The relation $K$ may assume quite different forms, depending on the choice of the microlocalization process. A natural approach is to cut-off $f$ by the translation at $x$ of a window function $g$ localized at the origin, and then consider the Fourier transform (short-time Fourier transform):
\begin{equation}\label{8.1}
	V_gf(x,\xi)=\int_{\rd} f(t)e^{-2\pi i\xi t}\overline{g(t-x)}dt, \qquad x,\xi\in\rd.
\end{equation}
A fundamental example is the wave front set of H\"ormander $WF(f)$, cf. \cite{B36}, where the support of $g$ shrinks to the origin and then, according to Heisenberg's uncertainty principle, the identification of the relevant frequencies is rough, namely $WF(f)$ is defined as a cone in the $\xi$ variables. Therefore, 
\begin{equation}\label{8.2}
	K:WF(f)\to WF(Tf)
\end{equation}
is represented by a geometric map.

In a more quantitative approach, one considers the short-time Fourier transform, given by the choice of a fixed window $g$ in \eqref{8.1}. The corresponding map $K$ is then a linear operator in $\rdd$ with kernel
\begin{equation}\label{8.3}
	k(z,w)=\la Tg_w,g_z\ra, \qquad z=(x,\xi), w=(y,\eta),
\end{equation}
\begin{equation}\label{8.4}
	g_z(t)=\pi(z)g(t)=M_\xi T_xg(t)=e^{2\pi i\xi t}g(t-x), \qquad t\in\rd.
\end{equation}
We used the notation in \cite{Elena-book, book}, where $k$ is called {\em Gabor matrix}. In this order of ideas, the more general setting was already given by Bony, cf. \cite{bony1,bony2,bony3}. Microlocalization there is provided by a partition of unity in terms of the atoms 
\begin{equation}\label{8.5}
	\psi_z(D)f, \qquad z=(x,\xi),
\end{equation}
where $\psi_z(D)$ are pseudodifferential operators with symbols localized at $z$ and satisfying uniform estimates in Weyl-H\"ormander classes. The corresponding microlocal representation is given by the norms $\norm{\psi_z(D)f}_{L^2(\rd)}$, whereas the corresponding map $K$ is identified by the kernel 
\begin{equation}\label{8.6}
	k(z,w)=\norm{\vartheta_z(D)T\psi_w(D)}_{\mathcal{L}(L^2(\rd))},
\end{equation}
for a suitable pair of families $\{\vartheta_z\}_z$, $\{\psi_w\}_w$.

In all the preceding constructions, pseudodifferential operators are characterized by the fact that the action of $K$ is concentrated at the diagonal, whereas in the case of FIOs it is concentrated at the graph of the related symplectic maps $\chi$. For the H\"ormander's wave front set, this is simply expressed by geometric inclusions, cf. \cite{B36}, for the Gabor matrix \eqref{8.3} and in the Bony's framework \eqref{8.6} by decay estimates for $k$. Actually, in the setting of the works by Bony, more general FIOs are defined in terms of decay estimates for the kernel $k$ outside a general manifold.

When comparing the Wigner microlocalization
\begin{equation}\label{8.7}
	Wf(x,\xi)=\int_{\rd}f(x+t/2)\overline{f(x-t/2)}e^{-2\pi i\xi t}dt,
\end{equation}
with \eqref{8.1}, we note that the role of the window $g$ in \eqref{8.1} is played in \eqref{8.7} by $f$ itself, and for symmetry reasons the translation at $x$ is distributed to both the factors under the integral. As already clear from the original paper by Wigner \cite{Wigner}, the main motivation for defining \eqref{8.7} is its striking effectiveness in the handling of Schr\"odinger equations. Besides, with respect to \eqref{8.1} and \eqref{8.5}, the Wigner distribution approach has the advantage that in \eqref{8.7} there is no dependence on the auxiliary functions $g$ and $\psi_z$. Nevertheless, the apparent drawback is the appearance of the so-called \emph{ghost frequencies}, see \cite[Section 8]{CGRPartII2022} for a comparison with the H\"ormander wave front set. On the other hand, the properties of the Wigner kernel in \eqref{nucleoFIOt} follow the lines of those valid for other microlocalizations. In fact, the Wigner kernel of pseudodifferential operators was proved in \cite{CRPartI2022} to be concentrated at the diagonal, and the results of the present paper, see in particular Remark \ref{e8}, clarify for FIOs concentration at the graph of the symplectic map. The algebraic setting of Definition \ref{C6T1.1old} and Theorem \ref{tc1old}, as well as our applications to Schr\"odinger equations, are inspired by the presentation by Bony, see \cite{bony1,bony2,bony3}. 

Surprisingly enough, ghost frequencies do not appear for the Wigner kernel in the case of pseudodifferential operators and FIOs with quadratic phase functions. In \cite{CGRPartIII}, we shall study the case of non-linear symplectic maps, where ghosts are still present, but they can be controlled by a smoothing argument.

Finally, concerning the global in time expression of the solutions of Schr\"odinger equations, the literature concerning the appearance of caustics is very large, let us limit to mention \cite{AS78, helffer84} proposing the use of Maslov index, and our recent paper \cite{CGRPartII2022} adopting the metaplectic framework. An alternative approach to the problem is given by semi-classical analysis, taking into account the role of Planck constant $\hbar$. According to Bohr's correspondence principle, the solution of a Schr\"odinger equation approaches the solutions of the Hamiltonian system as $\hbar$ becomes negligible. The quantitative formulation of this principle in large intervals of time has been extensively explored in the literature, under the name of {\em Ehrenfest time}. In this context, we mention \cite{bambusi, bouzouina} where the Hamiltonian is assumed to satisfy growth at infinity of quadratic type, cf. the preceding Section \ref{sec24}.

\begin{appendix}
	\section{}\label{A1}
	In this Appendix Section we prove the intertwining relation needed for the conclusion of Theorem \ref{stime}.
	\begin{proposition}\label{propA1}
	Let $M\in GL(2d,\bR)$ and $\mathfrak{T}_MF=F\circ M$, $F\in L^2(\rdd)$. Define
	\begin{equation}\label{commRel}
		\hat U=\cF_2^{-1}\mathfrak{T}_M\cF_2,
	\end{equation}
	i.e., $\hat U$ satisfies the intertwining relation $\mathfrak{T}_M\cF_2=\cF_2\hat U$, and $U=\pi^{Mp}(\hat U)\in Sp(2d,\bR)$. Then,\\
	(i) $U$ is upper triangular if and only if $M$ is upper triangular.\\
	(ii) $U$ is lower triangular, i.e., $B=0_{2d\times 2d}$ in its block decomposition \eqref{S}, if and only if $M$ is lower triangular.\\
	(iii) $U$ is diagonal, i.e., $B=C=0_{2d\times 2d}$ in its block decomposition \eqref{S}, if and only if $M$ is diagonal.
\end{proposition}
\begin{proof}
	Let us write
	\[
		M^{-1}=\begin{pmatrix}
			M_{11} & M_{12}\\
			M_{21} & M_{22}
		\end{pmatrix} \quad and \quad M^T=\begin{pmatrix}M'_{11} & M'_{12}\\
		M'_{21} & M'_{22}
		\end{pmatrix},
	\]
	for $M_{ij},M'_{ij}\in\bR^{d\times d}$ ($i,j=1,\ldots,2$), so that the matrix $\cD_M=\pi^{Mp}(\mathfrak{T}_M)$ is given by:
	\[
		\cD_M= \left(\begin{array}{cc|cc}
		M_{11} & M_{12} & 0 & 0\\
		M_{21} & M_{22} & 0 & 0\\
		\hline
		0 & 0 & M'_{11} & M'_{12}\\
		0 & 0 & M'_{21} & M'_{22}
	\end{array}\right).
	\]
	Since $\pi^{Mp}$ is a group homomorphism, at the level of projections formula \eqref{commRel} reads as
	\begin{equation}\label{projcommrel}
		\cA_{FT2}^{-1}\cD_M\cA_{FT2}=U,
	\end{equation}
	for $U\in Sp(2d,\bR)$, where $\cA_{FT2}$ is defined as in \eqref{AFT2}. Computing explicitly the left hand-side of \eqref{projcommrel} yields:
	\begin{equation}\label{prodint}
	U=\left(\begin{array}{cc|cc}
		M_{11} & 0 & 0 & -M_{12}\\
		0& M'_{22} & -M'_{21} & 0\\
		\hline
		0 & -M'_{12} & M'_{11} & 0\\
		-M_{21} & 0 & 0 & M_{22}
	\end{array}\right),
\end{equation}
which is upper triangular if and only if $M_{21}=M'_{12}=0$, i.e. if and only if $M^{-1}$ is upper triangular (and $M^T$ is lower triangular), if and only if $M$ is upper triangular. This proves $(i)$. Items $(ii)$ and $(iii)$ follow analogously.
\end{proof}

\end{appendix}
\section*{Acknowledgements}
The authors have been supported by the Gruppo Nazionale per l’Analisi Matematica, la Probabilità e le loro Applicazioni (GNAMPA) of the Istituto Nazionale di Alta Matematica (INdAM). Gianluca Giacchi is supported also by University of Bologna and HES-SO Valais-Wallis. We are also thankful to Reinhard Werner for inspiring the explanatory example represented in Figure \ref{fig:Wigner}, and the anonymous referee for the thoughtful revision of this work and the significant suggested improvements.

\bibliographystyle{abbrv}

\end{document}